\documentclass[11pt,leqno,twoside]{amsart}
\usepackage{amssymb,amsmath,amsthm}
\usepackage[T1]{fontenc}
\usepackage[utf8]{inputenc}
\DeclareUnicodeCharacter{0096}{--}
\DeclareUnicodeCharacter{0092}{'}
\DeclareUnicodeCharacter{0093}{``}
\DeclareUnicodeCharacter{0094}{''}
\usepackage{indentfirst,latexsym}
\usepackage{mathrsfs}
\usepackage{cite,enumitem,graphicx}
\usepackage{xcolor}

\usepackage[colorlinks=false,pdfborder={0 0 0},urlcolor=blue,hypertexnames=false,
citecolor=red,linkcolor=blue,linktocpage,pdfpagelabels,
bookmarksnumbered,bookmarksopen]{hyperref}
\usepackage[english]{babel}
\usepackage[a4paper,left=2.61cm,right=2.61cm,top=2.72cm,bottom=2.72cm]{geometry}

\makeatletter
\providecommand\@dotsep{5}
\def\listtodoname{List of Todos}
\def\listoftodos{\@starttoc{tdo}\listtodoname}
\makeatother

\numberwithin{equation}{section}
\newtheorem{theorem}{Theorem}[section]
\newtheorem{proposition}[theorem]{Proposition}
\newtheorem{definition}[theorem]{Definition}
\newtheorem{lemma}[theorem]{Lemma}
\newtheorem{corollary}[theorem]{Corollary}

\newtheorem{remark}{Remark}

\newcommand{\R}{\mathbb{R}}
\newcommand{\N}{\mathbb{N}}
\def\real{\mathbb{R}}

\def\dis{\displaystyle}

\title[Elliptic problems for a vectorial $p\&q$-Laplacian]
{Elliptic problems involving a diffusion operator of vectorial
$p\&q$-Laplacian type}

\author[M. L. M. Carvalho]{Marcos L. M. Carvalho}
\author[J. R. Santos Jr.]{Jo\~ao R. Santos J\'unior}
\author[H. W. G. Silva]{Harlen Wenderson G. Silva}

\address[M. L. M. Carvalho]
{{\newline\indent Instituto de Matem\'atica e Estat\'istica
\newline\indent
Universidade Dederal de Goi\'{a}s
\newline\indent
74001-970, Goi\^{a}nia-GO, Brazil}}
\email{marcos\_leandro\_carvalho@ufg.br}

\address[J. R. Santos Jr.]
{{\newline\indent Faculdade de Matem\'atica
\newline\indent
Instituto de Ci\^{e}ncias Exatas e Naturais
\newline\indent
Universidade Federal do Par\'a
\newline\indent
Avenida Augusto Corr\^{e}a 01, 66075-110, Bel\'em, PA, Brazil}}
\email{joaojunior@ufpa.br}

\address[H. W. G. Silva]
{{\newline\indent Programa de Doutorado em Matemática
\newline\indent
Instituto de Ci\^{e}ncias Exatas e Naturais
\newline\indent
Universidade Federal do Par\'a
\newline\indent
Avenida Augusto Corr\^{e}a 01, 66075-110, Bel\'em, PA, Brazil}}
\email{harlen.silva@icen.ufpa.br}

\thanks{Jo\~ao R. Santos was partially
supported by CNPq with grant 304340/2025-1, Brazil.
Harlen W. G. Silva was partially
supported by CAPES. Marcos L. M. Carvalho was partially
supported by CNPq with grant 300411/2025-1.}

\begin{document}
\enlargethispage{3pt}
\color{black}

\color{black}

\color{black}
\baselineskip 11pt

\begin{abstract}
We study a broad class of anisotropic Dirichlet problems driven by the
nonhomogeneous operator
\[
\mathcal L_Au=-\sum_{i=1}^N\partial_i\!\left(
a_i(|\partial_i u|^{p_i})|\partial_i u|^{p_i-2}
\partial_i u\right),
\]
where the functions $a_i$ satisfy structural growth and monotonicity
conditions of $(p_i,q_i)$-type.  This framework contains, among other
examples, the anisotropic $\vec p$-Laplacian, the orthotropic $p\&q$-Laplacian,
and anisotropic operators combining several directional growth regimes.  We
first establish a comparison principle adapted to this lack of homogeneity.
For power nonlinearities, including the singular range, we then obtain
nonexistence, existence, singular uniqueness, comparison alternatives, and multiplicity results
according to the position of the reaction exponent relative to the
anisotropic growth and critical exponents.  In particular, an intermediate
power regime exhibits nonexistence for small positive parameters and
two distinct nontrivial solutions for large parameters without additional
restrictions on the smallest effective exponent. At partial resonance,
$r=p_1<p_N$, we identify an exact positive parameter threshold: there are
no nontrivial solutions at or below it, whereas each parameter above it
admits infinitely many pairs of solutions converging strongly to zero.
We also treat concave--convex reactions with subcritical or
critical growth and prove the existence of nontrivial nonnegative weak
solutions in the corresponding parameter ranges.  The arguments combine a
convexity method of D\'iaz--Saa type, anisotropic Sobolev embeddings, direct
minimization, Ekeland's variational principle, mountain-pass techniques, and
concentration--compactness.
\end{abstract}

\color{black}

\color{black}
\bigskip

\color{black}

\color{black}
\maketitle\color{black}
\begin{center}
\begin{minipage}{12cm}
\tableofcontents
\end{minipage}
\end{center}

\color{black}

\color{black}
\section{Introduction}

\color{black}
Let $\Omega\subset\mathbb R^N$ be a bounded domain with smooth boundary.  In
this paper we study the Dirichlet problem
\begin{equation}\label{intro-problem}
\left\{
\begin{aligned}
\mathcal L_Au&=f(x,u) &&\text{in }\Omega,\\
u&=0 &&\text{on }\partial\Omega,
\end{aligned}
\right.
\end{equation}
driven by the quasilinear operator
\begin{equation}\label{intro-operator}
\mathcal L_Au
:=-\sum_{i=1}^{N}\partial_i
\left(
a_i\!\left(\left|\partial_i u\right|^{p_i}\right)
\left|\partial_i u\right|^{p_i-2}
\partial_i u
\right).
\end{equation}
Here $1<p_1\leq\cdots\leq p_N$, the functions
$a_i\colon[0,\infty)\to(0,\infty)$ determine the diffusion law in each
coordinate direction, and $f$ is a Carath\'eodory reaction term.  The central
feature of \eqref{intro-operator} is that two independent sources of
nonhomogeneity are present at the same time: the exponent may vary with the
direction, and the coefficient $a_i$ may change the effective growth of the
flux as the size of $|\partial_i u|$ changes.  We refer to
\eqref{intro-operator} as an anisotropic vectorial
$\vec p\&\vec q$-Laplacian-type operator, where ``vectorial'' describes the
vectors of directional exponents rather than a vector-valued unknown.

The class \eqref{intro-problem} lies at the intersection of two well-developed
theories.  If $a_i\equiv1$, then \eqref{intro-operator} becomes the anisotropic
$\vec p$-Laplacian
\[
-\sum_{i=1}^N\partial_i
\bigl(|\partial_i u|^{p_i-2}\partial_i u\bigr).
\]
Its natural setting is an anisotropic Sobolev space, whose directional
derivatives possess different integrability exponents.  The foundations of
these spaces and their embeddings go back to Nikolskii, Troisi, Trudinger,
and R\'akosn\'ik
\cite{Nik70,Troi69,Trudin,Rako79,Rako81}.  Existence, nonexistence,
compactness, critical growth, singular reactions, and extremal problems for
anisotropic equations have subsequently been studied in, among others,
\cite{FraGaKaw04,HamRako05,HamRako06,HamRak07,AlvesHam08,DM,
GelsonGioLeand,LeggatMiri}.  On the other hand, if the directional exponents
are constant, $p_i=p$ and $q_i=q$, and
$a_i(t)=1+t^{(q-p)/p}$, then
\[
\mathcal L_Au=-\sum_{i=1}^N\partial_i\bigl(|\partial_i u|^{p-2}\partial_i u+|\partial_i u|^{q-2}\partial_i u\bigr),
\]
the orthotropic $p\&q$-Laplacian. Related isotropic $p\&q$-Laplacian problems have motivated a
substantial literature concerning eigenvalues, multiplicity, sign-changing
solutions, and critical growth; see, for example,
\cite{HeLi,LiLiang,LiGuo,MuYang,GM13,BariFigue,BarFigue,A.F.J.R}.

The isotropic theory with a single constitutive function provides another
important antecedent. Its typical operator is
$-\operatorname{div}(\phi(|\nabla u|)\nabla u)$, with energy density
$\Phi(t)=\int_0^t s\phi(s)\,ds$. The Orlicz--Sobolev framework and its
embedding theory were developed in \cite{DonaldsonTrudinger1971,Gossez1974};
principal eigenvalues for general quasilinear operators were studied in
\cite{GarciaHuidobro1999}. Multiple positive solutions under assumptions on
the constitutive law and the reaction were obtained in
\cite{FukagaiNarukawa2007}, while critical Orlicz--Sobolev growth on
$\mathbb R^N$ was treated in \cite{FukagaiItoNarukawa2006}.
The generalized $p\&q$ problems in \cite{A.F.J.R,BarFigue} are particularly
close to the present use of a nonconstant diffusion coefficient.
For example, the radial flux $\phi(t)=a(t^p)t^{p-2}$ involves a single
function $a$. Nevertheless, even when all $p_i=p$ and all $a_i=a$,
\eqref{intro-operator} acts on the individual derivatives rather than on
$|\nabla u|$. This distinction separates the present orthotropic energy
from the radial $\Phi$-Laplacian and prevents a direct identification of
their variational geometries.

For the standard isotropic two-power operator, an early treatment of
stationary reaction--diffusion equations is \cite{CherfilsIlyasov2005};
the survey \cite{MaranoMosconi2018} describes existence, resonance, and
critical-growth questions for the Dirichlet problem. In the anisotropic
two-power setting, positive solutions were studied in
\cite{RazaniFigueiredo2023,Tavares2023}, and superlinear nonlinearities
without the Ambrosetti--Rabinowitz condition were considered in
\cite{AmorosoSciammettaWinkert2024}. These works concern sums of directional
power fluxes under their respective growth assumptions. A recent extension
to coupled anisotropic two-power systems is
\cite{RazaniTavaresSousa2025}, where existence and a second solution are
obtained under different assumptions on the reaction.
Our coefficients
$a_i$ allow more general directional constitutive laws, and the power
reaction is examined across the different small-amplitude orders.

For the homogeneous anisotropic operator, the multiplicity theory in
\cite{FigueiredoSantosSuarez2015} is a direct antecedent. Positive
eigenfunctions obtained by sub- and supersolutions are studied in
\cite{CianiFigueiredoSuarez2021}, and the combination of that method with
the mountain-pass theorem appears in \cite{DosSantosFigueiredoSilva2020}.
Genus methods yielding infinitely many anisotropic solutions are also
developed in \cite{RazaniFigueiredo2022}. Thus infinite multiplicity for
anisotropic equations is already part of the literature; the specific
issue addressed here is the exact threshold and accumulation at zero for
each fixed parameter above it in the partially resonant power problem.

Regularity constitutes a related but distinct line of research.
Fully anisotropic Sobolev inequalities and local boundedness of minimizers
are treated in \cite{CianchiSobolev2000,CianchiBounded2000}. Lipschitz
regularity for orthotropic energies with nonstandard growth, including
singular diffusion, is investigated in
\cite{BousquetBrasco2020,BousquetBrascoLeone2024}. For radial quasilinear
operators, structural regularity theory and boundedness results for the
$\Phi$-Laplacian are found in \cite{Lieberman1991,SilvaCarvalhoAlbuquerque2020}.
These results require their own hypotheses and are cited as background;
they are not invoked to assert additional regularity for every coefficient
allowed in \eqref{intro-operator}.

\color{black}
These two reductions also explain the modeling relevance of the present
framework.  A $p\&q$-growth energy describes a diffusion mechanism with two
different power responses: one regime may dominate at moderate gradients,
whereas the other becomes decisive at larger gradients.  It is therefore a
canonical prototype for nonstandard-growth energies and constitutive laws,
as developed in the regularity theory of $(p,q)$-growth problems
\cite{PMar}.  Anisotropy has a different physical meaning: the constitutive
response depends on the direction, as occurs in media with preferred axes or
direction-dependent transport.  Such diffusion mechanisms also arise in
direction-sensitive image processing, where anisotropic smoothing is used to
reduce noise while preserving relevant structures \cite{Veraetal}.  The
operator \eqref{intro-operator} combines these effects by allowing both the
growth exponents and the transition between growth regimes to vary from one
direction to another.

We now describe the structural assumptions that make this generality
analytically tractable. For every $i\in\{1,\ldots,N\}$, let 
\(
1<p_1\leq\cdots\leq p_N,
~
1<q_1\leq\cdots\leq q_N,
~
p_i<q_i,
\)
and set
\[
\sigma_i:=\frac{q_i-p_i}{p_i}.
\]
The functions $a_i\in C([0,\infty),(0,\infty))\cap C^1((0,\infty),(0,\infty))$ are assumed, as required by
the result under consideration, to satisfy the following conditions:
\begin{enumerate}[label=$(a_\arabic*)$]
\item\label{intro-a1}
There exist constants $0<\alpha_i\leq\beta_i$, $\theta_i>0$, and
$\zeta_i\geq0$ such that

\[
\alpha_i+H(\zeta_i)\theta_i t^{\sigma_i}
\leq a_i(t)\leq
\beta_i+H(\zeta_i)\theta_i t^{\sigma_i},
\qquad t\geq0,
\]

where $H(0)=0$ and $H(t)=1$ for $t>0$.

\item\label{intro-a2}
The function $a_i$ is nondecreasing on $[0,\infty)$.

\item\label{intro-a3}
The logarithmic growth is strictly below the upper $q_i$-phase, namely
\[
t a_i'(t)<\sigma_i a_i(t),\qquad t>0.
\]
\end{enumerate}

Condition~\ref{intro-a1} gives the coercive lower bound and yields
$p_i$-growth when $H(\zeta_i)=0$, and $(p_i,q_i)$-growth when
$H(\zeta_i)=1$.

Notice that the presence of $H(\zeta_i)$ also in the upper bound of
$(a_1)$ is essential for the consistency between the growth of the energy
and the effective exponent $\gamma_i$. Indeed, if $H(\zeta_i)=0$, then
$\gamma_i=p_i$. If the upper bound were allowed to contain a $q_i$-term,
one could have
\(
\partial_i u\in L^{p_i}(\Omega)\setminus L^{q_i}(\Omega),
\)
so that a term of the form
\[
\int_\Omega |\partial_i u|^{q_i}\,dx
\]
could be infinite, and the energy would not be well defined on
$W_0^{1,\vec{\gamma}}(\Omega)$. Thus, when $H(\zeta_i)=0$, the upper
bound must exclude the $q_i$-growth, so that the $i$-th direction has
effectively $p_i$-growth. On the other hand, when $H(\zeta_i)=1$, the
$q_i$-growth is active and $\gamma_i=q_i$.

Condition~\ref{intro-a2} provides the
monotonicity and convexity required by comparison and variational arguments.
Condition~\ref{intro-a3} separates the actual growth of $a_i$ from the
limiting $q_i$-phase.  If
$A_i(t)=\int_0^t a_i(s)\,ds$, these hypotheses yield the estimates that place
the energy naturally in an anisotropic Sobolev space with effective exponent
\[
\gamma_i=(1-H(\zeta_i))p_i+H(\zeta_i)q_i.
\]

The flexibility of the operator is visible already at the level of admissible
coefficients.

The following standard examples are admissible for every $p_i>1$
(with positive constants and suitable choices of
$\alpha_i,\beta_i,$ and $\theta_i$). The first one corresponds to
$H(\zeta_i)=0$, whereas the remaining examples correspond to
$H(\zeta_i)=1$:

\begin{itemize}
\item $a_i(t)=c_i$, which gives weighted anisotropic $p_i$-growth;
\item $a_i(t)=c_i+d_i t^{\sigma_i}$, which gives anisotropic
$(p_i,q_i)$-growth;
\item
\[
a_i(t)=c_i+d_i t^{\sigma_i}
+e_i\frac{t^{\rho_i}}{1+t^{\rho_i}},
\qquad 0<\rho_i\leq\sigma_i,
\]
which adds a bounded monotone transition to the $(p_i,q_i)$ law;
\item more generally,
\[
a_i(t)=c_i+d_i t^{\sigma_i}
+\sum_{j=1}^{m_i}e_{ij}\frac{t^{\rho_{ij}}}{1+t^{\rho_{ij}}},
\qquad 0<\rho_{ij}\leq\sigma_i,
\]
which combines several bounded directional transitions without changing the
leading $q_i$-growth.
\end{itemize}
The first choice recovers the anisotropic $\vec p$-Laplacian.  The second
produces, up to positive weights, the sum of the directional $p_i$- and
$q_i$-Laplacians.  If all $p_i=p$, all $q_i=q$, and the coefficients are
independent of $i$, this reduces to the orthotropic $p\&q$-Laplacian, whose fluxes depend on individual derivatives rather than on $|\nabla u|$.  The last two
examples show that the theory is not restricted to an exact sum of two
homogeneous operators: it also permits bounded monotone transitions while
preserving the prescribed leading growth.  Thus the functions $a_i$ encode a
family of operators rather than a single fixed differential law.

The price of this generality is substantial.  In contrast with a homogeneous
$p$-Laplacian, the energy associated with \eqref{intro-operator} has no single
scaling, and the natural Sobolev geometry depends simultaneously on all the
effective exponents $\gamma_i$.  Consequently, coercivity, compactness, and
the geometry of the energy change as the reaction exponent crosses the
directional thresholds. When $\overline\gamma<N$ and $\gamma_+<\gamma^*$,
growth at the harmonic Sobolev exponent $\gamma^*$ requires a concentration argument;
for singular reactions, the energy is not differentiable at the origin in
the usual sense; and comparison cannot be obtained merely by importing the
standard isotropic argument.  These difficulties require a combination of
anisotropic estimates, convexity, truncation, and variational methods.

A key ingredient in our variational analysis is the compactness property of
the principal operator. Under $(a_1)-(a_2)$, we prove that $\mathcal A$ is
continuous, strictly monotone and of type $(S_+)$; see Proposition
\ref{Splus}. In particular, if $u_n\rightharpoonup u$ in
$W_0^{1,\vec{\gamma}}(\Omega)$ and
\[
\limsup_{n\to\infty}
\langle\mathcal A(u_n),u_n-u\rangle\leq0,
\]
then $u_n\to u$ strongly in $W_0^{1,\vec{\gamma}}(\Omega)$.
This property plays a fundamental role in recovering strong convergence of
Palais--Smale sequences in both the subcritical and critical regimes.

Our first contribution is a comparison principle for nonnegative sub- and
supersolutions of \eqref{intro-problem}. Assuming
\ref{intro-a1}--\ref{intro-a2} and that
\[
t\longmapsto \frac{f(x,t)}{t^{p_1-1}}
\quad\text{is strictly decreasing for a.e. }x\in\Omega,
\]
we prove that $u_1\leq u_2$ a.e. in $\Omega$, provided that
$u_1\leq Cu_2$ a.e. in $\Omega$ for some $C>0$; see
Theorem~\ref{Comparison principle}. The proof is based on the convexity of a
functional constructed after the transformation $u\mapsto u^{p_1}$ and on
test functions of D\'iaz--Sa\'a type \cite{DiazSaa1987}. This result is not merely auxiliary: it
also supplies the comparison alternative in the regular sub-$p_1$ regime.
In the singular regime, strict monotonicity yields unconditional uniqueness.

We next consider the pure-power problem
\[
\mathcal L_Au=\lambda |u|^{r-2}u,
\qquad u\in W_0^{1,\vec\gamma}(\Omega).
\]
Let $\gamma^*$ denote the harmonic anisotropic Sobolev exponent and let
$\widetilde p$ be the largest active exponent among $p_N$ and the $q_i$-phases.
The results reveal a clear dependence on the relative position of $r$.  For
$\lambda\leq0$ and $1<r<\gamma^*$, the trivial solution is the only weak solution.  In the
intermediate range
$p_1<r<\widetilde p<\gamma^*$, there exists a positive threshold below which
no nontrivial solution exists, whereas sufficiently large values of
$\lambda$ yield a nontrivial solution; see
Theorem~\ref{Theorem p1 smallest r smallest critical range}. The same
large-parameter regime contains at least two distinct nontrivial solutions
at negative and positive energy levels; see
Theorem~\ref{palais smale r menor p tilde}. The full-modular estimate in
Lemma~\ref{lema r less critical and p tilde consider the cases} proves
the local minimum at zero throughout this range, without requiring
$\gamma_-<r$ or, when $\overline\gamma<N$, $\gamma_-<\gamma^*/N$.
Retaining all directional factors in the Sobolev product inequality is
essential to this improvement. 
In contrast, if $1<r<p_1$, a nontrivial nonnegative solution exists for
every $\lambda>0$. Moreover, two nonnegative solutions coincide whenever
they are mutually comparable up to positive constants. In particular, if
two distinct positive a.e. solutions exist, at least one of the corresponding
quotients is not essentially bounded.
The borderline case $r=p_1<p_N$ has a different geometry. Let
$Y=\{i:p_i=p_1\}$. Theorem~\ref{partial-resonance-infinite} determines
the exact threshold $\Lambda_Y(\Omega)$, defined in
\eqref{partial-rayleigh-threshold} using only the coefficients $a_i(0)$
and the derivatives in $Y$. No nontrivial solution exists for
$\lambda\leq\Lambda_Y(\Omega)$, while each fixed
$\lambda>\Lambda_Y(\Omega)$ admits infinitely many distinct pairs
of weak solutions converging strongly to zero. This sharpens, at partial
resonance and for the present nonhomogeneous operators, the
large-parameter finite-multiplicity conclusion for the homogeneous
anisotropic problem in \cite[Theorem 1.2]{FigueiredoSantosSuarez2015}.
For the homogeneous operator, \cite{CianiFigueiredoSuarez2021} establishes
a positive-eigenfunction threshold when $p_1\leq r<p_N$. At $r=p_1$, the
present theorem identifies the threshold by a directional Rayleigh
quotient, includes nonexistence at equality, and supplies a sequence of
weak solutions for every parameter strictly above it; positivity of that
sequence is not asserted.
The proof combines a sectionwise Rayleigh construction with the
Clark genus framework \cite{Clark1972} and its partially sublinear
extension \cite{LiuWangClark}, and does not require $(a_3)$.
Corollary~\ref{partial-disk-example} exhibits, on a smooth disk and with
linear reaction, an interval of parameters admitting infinitely many
anisotropic solutions although the corresponding Laplacian problem has
only zero. The effect comes from different small-amplitude orders in
different directions; the sign of the sequence is not prescribed.

For the singular range $0<r<1$, we obtain a nontrivial positive a.e. solution
for every $\lambda>0$, and this solution is unique;
see Theorem~\ref{Singular Case Theorem}.

The concave--convex mechanism originates, in the semilinear setting, in
\cite{AmbrosettiBrezisCerami1994}. Quasilinear extensions include
\cite{A2,GM13,CarvalhoM}; nonlinear Rayleigh quotients and Nehari-manifold
methods for radial $\Phi$-Laplacian problems are developed in
\cite{CarvalhoGasinskiSantosSilva2025}. These provide context for the
parameter-dependent competition of the reaction terms, while the
directional diffusion and critical exponent here are anisotropic.
Finally, we investigate concave--convex reactions of the form
\[
\mathcal L_Au=\lambda f(u)+g(u).
\]
For a concave term comparable with $t^{r-1}$, $1<r<p_1$, and a superlinear term
with admissible subcritical or critical growth, we obtain a nontrivial
nonnegative weak solution for sufficiently small positive $\lambda$.  We
also treat the critical choice
$g(t)=|t|^{\gamma^*-2}t$ under an
Ambrosetti--Rabinowitz-type condition on $f$ and prove existence for all
sufficiently large $\lambda$; see
Theorem~\ref{critical-concave-convex}.  The proofs use direct minimization and
Ekeland's variational principle \cite{Ekeland1974} in the concave regime, the mountain-pass
theorem \cite{AM} for multiplicity and superlinear reactions, and an anisotropic
concentration--compactness argument to recover compactness below the critical
energy threshold. The underlying concentration--compactness framework
goes back to \cite{LionsLimitI1985,LionsLimitII1985}; anisotropic
critical-growth problems and compactness are treated in
\cite{HamRako05,AlvesHam08}.
In this way, one structural setting yields a unified
treatment of comparison, nonexistence, existence, and multiplicity for
several operators that had previously been studied separately.

The paper is organized as follows.  Section~2 introduces the structural
assumptions, the effective anisotropic exponents, and the variational and
embedding tools used throughout the work.  Section~3 proves the comparison
principle.  Section~4 is devoted to power-type reactions: it contains the
nonexistence results, the existence and multiplicity theory in the regular
subcritical regimes, the exact threshold and infinite multiplicity at
partial resonance, and the singular case.  Section~5 treats concave--convex
nonlinearities, including the critical-growth problem.  The bibliography at
the end collects the anisotropic, $p\&q$-growth, and variational references
used in the analysis.

\section{{Assumptions and variational background}}
Since our approach is based on variational methods, this section is devoted to establishing a suitable variational framework for problem \eqref{P}. To be more precise, throughout the paper, we will be concerned with the following general class of problems

\begin{equation}\label{P}\tag{$P$}
\left \{ \begin{array}{ll}
  -\displaystyle\sum_{i=1}^{N}\partial_{i}\left(a_{i}\left(\left|\partial_{i} u\right|^{p_{i}}\right)\left|\partial_{i} u\right|^{p_{i}-2}\partial_{i} u\right)=f(x,u) \; \mbox{in} \; \Omega , \; \\
  u=0, \;\mbox{on}\; \partial\Omega,
\end{array}\right.
\end{equation}
where $\Omega\subset\R^N$, $N\geq1$, is a bounded smooth domain, and $f$ is a Carath\'eodory function to be specified later. Throughout the paper each $a_i$ is positive and belongs to $C([0,\infty))\cap C^1((0,\infty))$. Conditions $(a_1)$--$(a_3)$ are imposed only as explicitly required by each result. The flux $a_i(|t|^{p_i})|t|^{p_i-2}t$ is assigned the value zero at $t=0$.

\begin{proposition}
	Assume that $(a_1)-(a_3)$ hold. Define
$\displaystyle A_i(t)=\int_0^t a_i(s)\,ds$. Then, for all $t,\rho>0$
	\begin{eqnarray}
		0\leq\dfrac{a_i'(t^{p_i})t^{p_i}}{a_i(t^{p_i})}<\dfrac{q_i-p_i}{p_i}\qquad\mbox{and}\qquad
		1\leq \dfrac{a_i(t^{p_i})t^{p_i}}{A_i(t^{p_i})}<\dfrac{q_i}{p_i},\label{desi-f}\\		
		\min\left\{\rho,\rho^{\frac{q_i}{p_i}} \right\}A_i(t)\leq A_i(\rho t)\leq \max\left\{\rho,\rho^{\frac{q_i}{p_i}} \right\}A_i(t)\label{joao4 sugestao},\\
		\alpha_it+H(\zeta_i)\dfrac{p_i\theta_i}{q_i}t^{q_i/p_i}\leq A_i(t)\leq \beta_it+H(\zeta_i)\frac{p_i\theta_i}{q_i}t^{q_i/p_i}.\label{desig of item a1}
	\end{eqnarray}
\end{proposition}
\begin{proof}
	The first inequality follows immediately from $(a_2)$ and $(a_3)$.
	Since $a_i$ is nondecreasing,
	\[
	 A_i(s)=\int_0^s a_i(\tau)\,d\tau\leq s a_i(s),
	 \qquad s>0.
	\]
	On the other hand, $(a_3)$ gives
	\[
	 \tau a_i'(\tau)<\frac{q_i-p_i}{p_i}a_i(\tau).
	\]
	For $0<\varepsilon<s$, integrate over $(\varepsilon,s)$ and use integration by parts.
	Since $a_i$ is continuous at zero, $\varepsilon a_i(\varepsilon)\to0$ and
	$\int_0^\varepsilon a_i(\tau)\,d\tau\to0$. Letting $\varepsilon\downarrow0$, we obtain
	\[
	 s a_i(s)-A_i(s)
	 <\frac{q_i-p_i}{p_i}A_i(s),
	\]
	The strict inequality persists because the positive deficit in $(a_3)$ has a positive integral on a fixed compact subinterval of $(0,s)$. This proves
	\[
	 1\leq\frac{s a_i(s)}{A_i(s)}<\frac{q_i}{p_i}.
	\]

If $\rho\geq1$, integrating between $t$ and $\rho t$ gives
\[
\log\rho
\leq
\log\frac{A_i(\rho t)}{A_i(t)}
\leq
\frac{q_i}{p_i}\log\rho,
\]
and therefore
\[
\rho A_i(t)
\leq A_i(\rho t)
\leq
\rho^{q_i/p_i}A_i(t).
\]
If $0<\rho<1$, the same argument gives
\[
\rho^{q_i/p_i}A_i(t)
\leq A_i(\rho t)
\leq
\rho A_i(t).
\]
Thus \eqref{joao4 sugestao} follows.

Finally, integrating the two inequalities in $(a_1)$ over $(0,t)$,
and using
\(
\sigma_i+1=\frac{q_i}{p_i},
\)
we obtain
\[
\alpha_i t+
H(\zeta_i)\frac{p_i\theta_i}{q_i}t^{q_i/p_i}
\leq A_i(t)\leq
\beta_i t+
H(\zeta_i)\frac{p_i\theta_i}{q_i}t^{q_i/p_i},
\]
which proves \eqref{desig of item a1}.
\end{proof}

The estimate \eqref{desig of item a1} follows from $(a_1)$ alone and will also be used in results that do not assume $(a_3)$.

\color{black}
Throughout this section we will introduce the main definitions and properties of anisotropic Sobolev spaces that will be used in this article. For a more detailed and thorough approach to the subject, we refer the interested reader to the references \cite{Rako79} and \cite{Rako81}.

Given $\overrightarrow p=(p_1,\ldots,p_N)$, assume
\begin{equation}\label{exponent definitions}
1<p_i<\infty,\qquad i=1,\ldots,N.
\end{equation}
For the structural exponents we retain $p_1\leq\cdots\leq p_N$, but the
embedding statements below do not require an ordering. Set
\[
\begin{gathered}
\overline p:=N\left(\sum_{i=1}^N\frac1{p_i}\right)^{-1},
\qquad p_+:=\max_i p_i,\\
p^*:=\begin{cases}
N\overline p/(N-\overline p),&\overline p<N,\\
+\infty,&\overline p\geq N,
\end{cases}
\qquad p_\infty:=\max\{p^*,p_+\}.
\end{gathered}
\]
An infinite endpoint is a convention for the ranges of finite exponents;
it is not used as a real power or as an assertion of an $L^\infty$ embedding.
The \emph{anisotropic Sobolev space} is defined by 
$$W^{1,\overrightarrow{p}}(\Omega):=\left\{v\in W^{1,1}(\Omega):\displaystyle\partial_i v\in L^{p_i}(\Omega)\right\}$$
and $W_0^{1,\overrightarrow{p}}(\Omega):=W^{1,\overrightarrow{p}}(\Omega)\cap W_0^{1,1}(\Omega)$ with the norm
\begin{equation}\label{norm 1 sobolev space}
|| u||_{1,\overrightarrow{p}}= \displaystyle\sum_{i=1}^N \left| \displaystyle\partial_i u \right|_{p_i}.
\end{equation}

\begin{proposition}
For every $r\geq1$, every $i\in\{1,\ldots,N\}$, and every
$u\in W_0^{1,1}(\Omega)$ such that $\partial_i u\in L^r(\Omega)$, the
following Poincar\'e-type inequality holds:
\begin{equation}\label{API}
|u|_{r}\leq \frac{dr}{2}\left| \partial_{i} u\right|_{r},
\end{equation}
where $d:=\sup_{x, y\in \Omega}\langle x-y, e_{i}\rangle$ and $e_i$ denotes any vector belonging to the canonical basis of $\R^{N}$.
\end{proposition}
\begin{proof}
	See the proof of the directional Poincar\'e inequality in \cite[Section 3, equation (11), pp.~722--723]{FraGaKaw04}.
\end{proof}
With the classical definition of the norm that describes anisotropic Sobolev spaces, we state a sequence of results that will be important throughout this article.

\begin{theorem}\label{FraGaKow04}
Suppose that $\Omega$ is a bounded smooth domain of $\R^N$ and
\eqref{exponent definitions} holds. Then $W_0^{1,\vec p}(\Omega)$ is a
reflexive Banach space. If $\overline p<N$, it is continuously embedded
in $L^r(\Omega)$ for every $1\leq r\leq p_\infty$, and the embedding is
compact for $1\leq r<p_\infty$. If $\overline p\geq N$, it is continuously
and compactly embedded in every $L^r(\Omega)$ with $1\leq r<\infty$.
\end{theorem}
\begin{proof}
	For $\overline p<N$, see \cite[Theorem 1, p.~717]{FraGaKaw04}; the proof is unchanged by a permutation of the directional exponents. If $N=1$, the fundamental theorem of calculus and H\"older's inequality give an $L^\infty$ bound and equicontinuity on the interval $\Omega$, hence compactness in every finite $L^r$. For $N\geq2$, $\overline p\geq N$ and a fixed finite $r$, choose $1<s_i<p_i$ with $\overline s<N$ and $s^*>r$, apply the finite-exponent result to $\vec s$, and use H\"older's inequality on the bounded domain.
\end{proof}

\color{black}
The proof of the theorem below can be seen at \cite{Troi69}:

\begin{theorem}\label{anisotropic inequality}
If $\overline p<N$, there is a constant $C_T>0$, depending only on
$N$ and $\vec p$, such that, for every $u\in W_0^{1,\vec p}(\Omega)$,
\begin{itemize}
\item[(i)] $\displaystyle |u|_{p^*}\leq C_T\prod_{i=1}^N|\partial_i u|_{p_i}^{1/N}$;
\item[(ii)] $\displaystyle |u|_{p^*}\leq\frac{C_T}{N}\sum_{i=1}^N|\partial_i u|_{p_i}$.
\end{itemize}
If $\overline p\geq N$, then for every fixed $1\leq r<\infty$ there is
$C_{r,\Omega,\vec p}>0$ such that
\[
|u|_r\leq C_{r,\Omega,\vec p}\prod_{i=1}^N|\partial_i u|_{p_i}^{1/N}.
\]
\end{theorem}

\begin{proof}
The first inequality is Troisi's inequality, and the second follows by
the arithmetic--geometric mean inequality. If $N=1$, the remaining assertion follows from the one-dimensional $L^\infty$ bound and H\"older's inequality. If $N\geq2$,
choose $\vec s$ as in the preceding proof, apply Troisi at $s^*>r$, and
use H\"older's inequality to replace the directional $s_i$-norms by the
$p_i$-norms and the $s^*$-norm of $u$ by its $r$-norm. This explains the
dependence on $r$ and $\Omega$.
\end{proof}

We denote $\overrightarrow{\gamma}=(\gamma_1,\ldots,\gamma_N) \in \R^N$, imposing the following condition for each $1\leq i\leq N$:
$$\gamma_i= (1-H(\zeta_i))p_i+ H(\zeta_i)q_i.$$

We also set
\[
\gamma_-:=\min_{1\leq i\leq N}\gamma_i
\qquad\mbox{and}\qquad
\gamma_+:=\max_{1\leq i\leq N}\gamma_i.
\]

Set $\mathscr I_q:=\{i:H(\zeta_i)=1\}$. In all sums of $q_i$-norms below, only indices in $\mathscr I_q$ are included; a norm in an inactive phase is never evaluated. For several results contained in this work, it will be useful to consider the following equivalent norm on the space $W_0^{1,\overrightarrow{\gamma}}(\Omega)$ by taking the following definition:

\begin{equation}\label{norm 2 sobolev space}
||u||:= ||u||_{1,\overrightarrow{p}}+\displaystyle\sum_{i\in\mathscr I_q}\left|\partial_i u\right|_{q_i}
\end{equation}
where $\overrightarrow{p}=(p_1,p_2,...,p_N)$. Since $\Omega$ is bounded and $q_i>p_i$, we have
$L^{q_i}(\Omega)\hookrightarrow L^{p_i}(\Omega)$ for every
$i\in\mathscr I_q$. Hence the norm in \eqref{norm 2 sobolev space}
is equivalent to the standard norm of
$W_0^{1,\overrightarrow{\gamma}}(\Omega)$.
Whenever there is no possibility of confusion, we will denote $(W_0^{1,\overrightarrow{\gamma}}(\Omega),||\cdot||)$ simply by $W_0^{1,\overrightarrow{\gamma}}(\Omega)$. 

Combining \eqref{desig of item a1} with the definition of the effective
exponents $\gamma_i$, we see that the energy associated with the principal
part is well defined on $W_0^{1,\overrightarrow{\gamma}}(\Omega)$.
More precisely, for every
$u\in W_0^{1,\overrightarrow{\gamma}}(\Omega)$,
\begin{equation}\label{energy-growth}
\begin{aligned}
\sum_{i=1}^{N}\frac{\alpha_i}{p_i}
\left\|\partial_i u\right\|_{p_i}^{p_i}
+
\sum_{i=1}^{N}
H(\zeta_i)\frac{\theta_i}{q_i}
\left\|\partial_i u\right\|_{q_i}^{q_i}
&\leq
\sum_{i=1}^{N}\frac1{p_i}
\int_\Omega
A_i\left(
\left|\partial_i u\right|^{p_i}
\right)\,dx
\\
&\leq
\sum_{i=1}^{N}\frac{\beta_i}{p_i}
\left\|\partial_i u\right\|_{p_i}^{p_i}
+
\sum_{i=1}^{N}
H(\zeta_i)\frac{\theta_i}{q_i}
\left\|\partial_i u\right\|_{q_i}^{q_i}.
\end{aligned}
\end{equation}
In particular, if $H(\zeta_i)=0$, then $\gamma_i=p_i$ and only the
$p_i$-growth is present in the $i$-th direction, whereas if
$H(\zeta_i)=1$, then $\gamma_i=q_i$ and both the $p_i$- and
$q_i$-terms are controlled by the working space.

For later use in the variational arguments, we establish the following compactness property of the principal operator. The proof follows the standard $(S_+)$ argument for Leray--Lions operators developed in \cite{LeSplus}, together with its anisotropic version given in \cite{Boureanu2023}.

\begin{proposition}\label{Splus}
Assume that $(a_1)-(a_2)$ hold. Then the operator
$\mathcal A:W_0^{1,\vec{\gamma}}(\Omega)
\to\left(W_0^{1,\vec{\gamma}}(\Omega)\right)^*$ defined by
\[
\langle\mathcal A(u),v\rangle
=
\sum_{i=1}^{N}\int_\Omega
a_i(|\partial_i u|^{p_i})
|\partial_i u|^{p_i-2}\partial_i u\,
\partial_i v\,dx
\]
is continuous, strictly monotone and of type $(S_+)$.
\end{proposition}
\begin{proof}
Consider
\[
\Psi(u):=
\sum_{i=1}^{N}\frac{1}{p_i}
\int_{\Omega}
A_i\left(\left|\partial_i u\right|^{p_i}\right)dx.
\]
By \eqref{desig of item a1}, $\Psi$ is well defined in
$W_0^{1,\vec{\gamma}}(\Omega)$. Moreover, if
\(
h_i(t)=\frac{1}{p_i}A_i(|t|^{p_i}),
\)
then, for $t\neq0$,
\[
h_i''(t)
=
p_i a_i'(|t|^{p_i})|t|^{2p_i-2}
+
(p_i-1)a_i(|t|^{p_i})|t|^{p_i-2}>0.
\]
The derivative $h_i'(t)=a_i(|t|^{p_i})|t|^{p_i-2}t$ extends continuously through zero with value zero. Its monotonicity on both open half-lines therefore extends across zero. Hence, by $(a_2)$ and the positivity of $a_i$, $\Psi$ is convex. No derivative of $a_i$ at zero is used.
Furthermore, from $(a_1)$,
\[
a_i(|t|^{p_i})|t|^{p_i-1}
\leq
\beta_i|t|^{p_i-1}
+
H(\zeta_i)\theta_i|t|^{q_i-1},
\]
and therefore
$\Psi\in C^1(W_0^{1,\vec{\gamma}}(\Omega),\R)$, with
\[
\langle\Psi'(u),v\rangle
=
\sum_{i=1}^{N}\int_{\Omega}
a_i(|\partial_i u|^{p_i})
|\partial_i u|^{p_i-2}\partial_i u\,
\partial_i v\,dx.
\]
Thus $\Psi'=\mathcal A$. By the convexity of $\Psi$ and
\cite[Proposition 42.6]{Zeidler}, $\mathcal A$ is monotone.

Now, let $(u_n)\subset W_0^{1,\vec{\gamma}}(\Omega)$ be such that
\(
u_n\rightharpoonup u
~\mbox{in }W_0^{1,\vec{\gamma}}(\Omega)
\)
and
\[
\limsup_{n\to\infty}
\langle\mathcal A(u_n),u_n-u\rangle\leq0.
\]
Since $\mathcal A(u)\in
(W_0^{1,\vec{\gamma}}(\Omega))^*$, we have
\[
\langle\mathcal A(u),u_n-u\rangle\to0.
\]
It follows that
\[
\limsup_{n\to\infty}
\langle\mathcal A(u_n)-\mathcal A(u),u_n-u\rangle\leq0.
\]
On the other hand, by the monotonicity of $\mathcal A$,
\[
\langle\mathcal A(u_n)-\mathcal A(u),u_n-u\rangle\geq0.
\]
Therefore,
\begin{equation}\label{Splus-limit}
\langle\mathcal A(u_n)-\mathcal A(u),u_n-u\rangle\to0.
\end{equation}

For $i\in\{1,\ldots,N\}$, put
\(
b_i(t)=a_i(|t|^{p_i})|t|^{p_i-2}t.
\)
From $(a_1)$ and $(a_2)$, for $t\neq0$,
\[
\begin{aligned}
b_i'(t)
&=
p_i a_i'(|t|^{p_i})|t|^{2p_i-2}
+
(p_i-1)a_i(|t|^{p_i})|t|^{p_i-2}
\\
&\geq
(p_i-1)\alpha_i|t|^{p_i-2}
+
H(\zeta_i)(p_i-1)\theta_i|t|^{q_i-2}.
\end{aligned}
\]
Hence, the function
\[
t\mapsto
b_i(t)-\alpha_i|t|^{p_i-2}t
-
H(\zeta_i)\frac{(p_i-1)\theta_i}{q_i-1}
|t|^{q_i-2}t
\]
is nondecreasing on each open half-line and continuous at zero, so it is nondecreasing on $\R$. Therefore, for every $\xi,\eta\in\R$,
\begin{equation}\label{bi-monotonicity}
\begin{aligned}
\bigl[b_i(\xi)-b_i(\eta)\bigr](\xi-\eta)
&\geq
\alpha_i
\left(
|\xi|^{p_i-2}\xi-|\eta|^{p_i-2}\eta
\right)(\xi-\eta)
\\
&\quad+
H(\zeta_i)\frac{(p_i-1)\theta_i}{q_i-1}
\left(
|\xi|^{q_i-2}\xi-|\eta|^{q_i-2}\eta
\right)(\xi-\eta).
\end{aligned}
\end{equation}
Since the first term on the right-hand side of \eqref{bi-monotonicity} is strictly positive whenever
$\xi\ne\eta$, it follows that $\mathcal A$ is strictly monotone.
Using \eqref{Splus-limit} and \eqref{bi-monotonicity}, we obtain
\begin{equation}\label{p-monotonicity-limit}
\begin{aligned}
\int_{\Omega}
\left(
|\partial_i u_n|^{p_i-2}\partial_i u_n
-
|\partial_i u|^{p_i-2}\partial_i u
\right)
(\partial_i u_n-\partial_i u)\,dx
\to0
\end{aligned}
\end{equation}
for every $i\in\{1,\ldots,N\}$. Moreover, if $H(\zeta_i)=1$, then
\begin{equation}\label{q-monotonicity-limit}
\begin{aligned}
\int_{\Omega}
\left(
|\partial_i u_n|^{q_i-2}\partial_i u_n
-
|\partial_i u|^{q_i-2}\partial_i u
\right)
(\partial_i u_n-\partial_i u)\,dx
\to0.
\end{aligned}
\end{equation}

For $m>1$, we have
\[
\left(
|\xi|^{m-2}\xi-|\eta|^{m-2}\eta
\right)(\xi-\eta)
\geq
\begin{cases}
C|\xi-\eta|^m, & m\geq2,\\[2mm]
\displaystyle
C\frac{|\xi-\eta|^2}
{(|\xi|+|\eta|)^{2-m}}, & 1<m<2.
\end{cases}
\]
For $1<m<2$, the quotient is defined as zero when $\xi=\eta=0$.
Thus, if $m\geq2$, \eqref{p-monotonicity-limit} and
\eqref{q-monotonicity-limit} give directly the corresponding strong
convergence in $L^m(\Omega)$.

If $1<m<2$, we have
\[
\int_{\Omega}
\frac{|v_n-v|^2}
{(|v_n|+|v|)^{2-m}}\,dx\to0,
\]
where $v_n=\partial_i u_n$ and $v=\partial_i u$. By H\"older's
inequality,
\[
\begin{aligned}
\int_{\Omega}|v_n-v|^m\,dx
&\leq
\left(
\int_{\Omega}
\frac{|v_n-v|^2}
{(|v_n|+|v|)^{2-m}}\,dx
\right)^{m/2}
\\
&\quad\times
\left(
\int_{\Omega}(|v_n|+|v|)^m\,dx
\right)^{(2-m)/2}.
\end{aligned}
\]
Since $(u_n)$ is bounded in
$W_0^{1,\vec{\gamma}}(\Omega)$, it follows that
\[
v_n\to v\quad\mbox{in }L^m(\Omega).
\]
Hence,
\[
\partial_i u_n\to\partial_i u
\quad\mbox{in }L^{p_i}(\Omega),
\qquad i=1,\ldots,N,
\]
and, whenever $H(\zeta_i)=1$,
\[
\partial_i u_n\to\partial_i u
\quad\mbox{in }L^{q_i}(\Omega).
\]
Therefore,
\[
u_n\to u
\quad\mbox{in }W_0^{1,\vec{\gamma}}(\Omega).
\]
Thus $\mathcal A$ is of type $(S_+)$.
\end{proof}

\section{{Comparison principle}}
In this section, we adapt arguments concerning sub- and supersolutions to the present class of anisotropic operators, following some of the technical ideas used by Carvalho et al. \cite{CarvalhoM} in the proof of their comparison principle.

When $a_i\equiv1$, Alves and El Hamidi established a comparison principle in \cite{AlvesHam08}, based on ideas from \cite{MW}.
\begin{definition}
Let $u\in W^{1,\vec{\gamma}}(\Omega)$ be nonnegative. We say that $u$ is a weak subsolution (supersolution) of problem \eqref{P} if $f(x,u(x))$ is well defined for a.e. $x\in\Omega$, \(f(x,u)\varphi\in L^1(\Omega)\) for every nonnegative $\varphi\in W_0^{1,\vec{\gamma}}(\Omega)$, and
\[
\sum_{i=1}^{N}\int_{\Omega}
a_i(|\partial_i u|^{p_i})
|\partial_i u|^{p_i-2}\partial_i u\,\partial_i\varphi\,dx
\leq \int_{\Omega}f(x,u)\varphi\,dx
\]
for a subsolution, with the reverse inequality for a supersolution. When $f(x,t)$ is defined only for $t>0$, as in singular problems, the requirement that $f(x,u)$ be well defined a.e. entails $u>0$ a.e. in $\Omega$. No value of $f(x,0)$ is prescribed.
\end{definition}

\begin{theorem}\label{Comparison principle}
Suppose that $(a_1)-(a_2)$ hold and that the map
\[
t\in(0,+\infty)\longmapsto
\frac{f(x,t)}{t^{p_1-1}}
\]
is strictly decreasing for a.e. $x\in\Omega$. Let $u_1,u_2\in W_0^{1,\vec{\gamma}}(\Omega)$ be nonnegative weak sub- and supersolutions of problem \eqref{P}, respectively. Assume that there exists $C>0$ such that \(u_1\leq Cu_2~\mbox{a.e. in }\Omega.\) Then \(u_1\leq u_2~\mbox{a.e. in }\Omega.\)
\end{theorem}
\begin{proof}
Initially,  we consider the functional $J:L^1(\Omega)\to (-\infty, +\infty]$ defined by 
\begin{equation}\label{problem nonlinear power part}
 J(u)=\left\{ \begin{array}{ll}
         \displaystyle\sum_{i=1}^{N}\frac{1}{p_i}\displaystyle\int_{\Omega}A_i\left(\left|\partial_i(u^{\frac{1}{p_1}})\right|^{p_i}\right)dx, &\mbox{if}~u\geq 0\, \,\mbox{and}\,\,\displaystyle u^{\frac{1}{p_1}}\in W_0^{1,\vec{\gamma}}(\Omega) ; \\
        \infty, & \mbox{otherwise}.\end{array} \right.
\end{equation}
We shall prove that the effective domain of $J$, i.e. the set
$$D(J):=\{u\in L^1(\Omega): u\geq 0 \,\,\mbox{and}\,\,\displaystyle u^{\frac{1}{p_1}}\in W_0^{1,\vec{\gamma}}(\Omega)\},$$
is non-empty, i.e., the functional $J\not\equiv \infty$. For this purpose, let $x_0\in\Omega$ and $\varepsilon>0$ be such that $B_\varepsilon(x_0)\subset \Omega$. Consider the function
\begin{equation}
 v_\varepsilon(x)=\left\{ \begin{array}{ll}
 \overline{v_\varepsilon}^{p_1+1}(|x-x_0|) & x\in B_\varepsilon(x_0) ;\\
 0, & x\in \Omega \setminus B_\varepsilon(x_0).\end{array} \right.
\end{equation}
where $\overline{v_\varepsilon}:[0,\varepsilon]\to\R$ is defined by 
\begin{equation}
 \overline{v_\varepsilon}(t)=\left\{ \begin{array}{ll}
 1 ,& t=0;\\
 \mbox{linear}, & 0<t<\displaystyle\frac{\varepsilon}{2};\\
 0, & \displaystyle\frac{\varepsilon}{2}\leq t\leq \varepsilon.\end{array} \right.
\end{equation}

It follows from the Inequality (\ref{desig of item a1}) that 
\begin{eqnarray}	
	J(v_\varepsilon)=\sum_{i=1}^{N}\frac{1}{p_i}\displaystyle\int_{\Omega}A_i\left(\left|\partial_i (v_\varepsilon^{\frac{1}{p_1}})\right|^{p_i}\right)dx &=& \sum_{i=1}^{N}\frac{1}{p_i}\displaystyle\int_{\Omega}A_i\left(\left(\frac{p_1+1}{p_1}\right)^{p_i}\overline{v_\varepsilon}^{\left(\frac{p_1+1}{p_1}- 1\right)p_i} \left|\partial_i\overline{v}_\varepsilon\right|^{p_i}\right)dx\nonumber\\
	&\leq& \nonumber \sum_{i=1}^{N}\frac{\beta_i}{p_i}\left( \frac{p_1+1}{p_1}\right)^{p_i}\int_\Omega \overline{v}_\varepsilon^{\frac{p_i}{p_1}}\left| \partial_i \overline{v}_\varepsilon \right|^{p_i}dx\\
	&+&\nonumber \sum_{i=1}^{N}\frac{\theta_i}{q_i}\left(\frac{p_1+1}{p_1}\right)^{q_i}H(\zeta_i)\int_{\Omega}\overline{v}_\varepsilon^{\frac{q_i}{p_1}}\left| \partial_i\overline{v}_\varepsilon\right|^{q_i}dx.
\end{eqnarray}
Since $0\leq \overline{v}_\varepsilon\leq 1$, we conclude that there exists $c>0$ such that
$$
J(v_\varepsilon)\leq c \sum_{i=1}^{N}\left( \int_{\Omega}\left| \partial_i\overline{v}_\varepsilon\right|^{p_i}dx +\int_{\Omega}\left| \partial_i\overline{v}_\varepsilon\right|^{q_i}dx \right) <\infty.
$$
Consequently, $D(J)\neq\emptyset$.

Next, we shall prove that $J$ is convex. Let  $w_1,w_2\in D(J)$. Define $z_j=w_j^{1/p_1},~j=1, 2$, and
$$z_3=(tw_1+(1-t)w_2)^{1/p_1}$$
where $t\in(0,1)$. We have
$$
\displaystyle\partial_i z_3=
\dfrac{1}{p_1}z_3^{1-p_1}\left(t\partial_i w_1+(1-t)\partial_i w_2\right), \;\mbox{a.e. in}\; [z_3>0].
$$
Consequently,
$$
z_3^{p_1-1}\left| \partial_i z_3 \right|\leq \frac{t}{p_1}\left| \partial_i w_1 \right|+ \frac{(1-t)}{p_1}\left| \partial_i w_2 \right|.
$$
Since
$$
\partial_i w_j= p_1z_j^{p_1-1}\partial_i z_j, j=1,2,
$$
we obtain
$$
z_3^{p_1-1}\left|\partial_i z_3\right|\leq tz_1^{p_1-1}\left| \partial_i z_1 \right|+(1-t)z_2^{p_1-1}\left| \partial_i z_2 \right|.
$$
We can write
$$
z_3^{p_1-1}\left|\partial_i z_3\right|\leq \left( t^{\frac{p_1-1}{p_1}}z_1^{p_1-1}\right)t^{\frac{1}{p_1}}\left| \partial_i z_1\right|+\left((1- t)^{\frac{p_1-1}{p_1}}z_2^{p_1-1}\right)(1-t)^{\frac{1}{p_1}}\left| \partial_i z_2\right|.
$$
Using H\"older's inequality in $\R^2$ with conjugate exponents $\frac{p_1}{p_1-1}$ and $p_1$, where $|(x,y)|_{p_1}=\left(|x|^{p_1}+|y|^{p_1}\right)^{1/p_1}$, to the vectors
$$
\left( (t^{1/p_1}z_1)^{p_1-1}, ((1-t)^{1/p_1}z_2)^{p_1-1}\right)\,\,\mbox{and}\,\, \left( t^{1/p_1}\left|\partial_i z_1\right|, (1-t)^{1/p_1}\left| \partial_i z_2\right| \right)
$$
we deduce that
$$
\begin{array}{lll}
	z_3^{p_1-1}\left| \partial_i z_3\right|&\leq& [(t^{1/p_1}z_1)^{p_1}+((1-t)^{1/p_1}z_2)^{p_1}]^{\frac{p_1-1}{p_1}}. \left[t\left| \partial_i z_1 \right|^{p_1}+(1-t)\left| \partial_i z_2 \right|^{p_1}\right]^{1/p_1}\\
	&=& [tw_1+(1-t)w_2]^{\frac{p_1-1}{p_1}}.\left[ t\left| \partial_i z_1 \right|^{p_1}+(1-t)\left| \partial_i z_2 \right|^{p_1} \right]^{1/p_1}\\
	&=& z_3^{p_1-1}\left[ t\left| \partial_i z_1 \right|^{p_1}+(1- t)\left| \partial_i z_2 \right|^{p_1} \right]^{\frac{1}{p_1}}.
\end{array}
$$
This implies that 
\begin{equation}\label{inequality john 5}
\left| \partial_i z_3\right|\leq \left[ t\left| \partial_i z_1 \right|^{p_1}+(1-t)\left| \partial_i z_2 \right|^{p_1} \right]^{\frac{1}{p_1}}.
\end{equation}

Next, we claim that $t\mapsto A_i(t^{p_i/p_1})$ is convex. Indeed,
\[
\frac{d}{dt}A_i(t^{p_i/p_1})
=\frac{p_i}{p_1}a_i(t^{p_i/p_1})
t^{(p_i-p_1)/p_1}.
\]
Since $p_i\geq p_1$, the power $t^{(p_i-p_1)/p_1}$ is nondecreasing; the other factor is nondecreasing by $(a_2)$. For $t>0$ the derivative above is nondecreasing. It has a finite right limit at zero, equal to $a_i(0)$ if $p_i=p_1$ and to zero if $p_i>p_1$. The claimed convexity therefore holds on $[0,\infty)$ for every $p_i>1$, without assuming $a_i'(0)$ exists.

Finally, we have 
$$
J(tw_1+(1-t)w_2)= J(z_3^{p_1})= 
\sum_{i=1}^{N}\dfrac{1}{p_i}\int_{\Omega}A_i\left( \left| \partial_i z_3 \right|^{p_i}\right)dx.
$$
Using \eqref{inequality john 5} and the fact that $A_i$ is non-decreasing, we have
$$
J(tw_1+(1-t)w_2)\leq \displaystyle\sum_{i=1}^{N}\frac{1}{p_i}\int_\Omega A_i\left( \left[ t\left| \partial_i z_1 \right|^{p_1}+(1- t)\left| \partial_i z_2 \right|^{p_1} \right]^{\frac{p_i}{p_1}}\right)dx.
$$
Given the convexity of the function $t\mapsto A_i(t^{\frac{p_i}{p_1}})$, it follows that
\[
\begin{aligned}
J(tw_1+(1-t)w_2)
&\leq t\sum_{i=1}^N\frac1{p_i}\int_\Omega
A_i\bigl(|\partial_i(w_1^{1/p_1})|^{p_i}\bigr)\,dx\\
&\quad +(1-t)\sum_{i=1}^N\frac1{p_i}\int_\Omega
A_i\bigl(|\partial_i(w_2^{1/p_1})|^{p_i}\bigr)\,dx\\
&=tJ(w_1)+(1-t)J(w_2).
\end{aligned}
\]
Therefore, $J$ is a convex functional.

Now, suppose that the set
\[
\Omega_0=\{x\in\Omega:u_1(x)>u_2(x)\}
\]
has positive Lebesgue measure. Since $u_1,u_2\geq0$ and $u_1\leq Cu_2$ a.e. in $\Omega$, we have \(u_1>u_2>0~\mbox{a.e. in }\Omega_0.\) We define, for $j=1,2$,
\[
\varphi_j=
\begin{cases}
\dfrac{u_1^{p_1}-u_2^{p_1}}{u_j^{p_1-1}},&\mbox{in }\Omega_0,\\[2mm]
0,&\mbox{in }\Omega\setminus\Omega_0.
\end{cases}
\]

We will prove that $\varphi_j\in W_0^{1,\vec{\gamma}}(\Omega)$. In fact, note that $\varphi_1=\varphi_2=0$ a.e. in $\Omega\setminus\Omega_0$ and $\varphi_1=u_1-\left( \dfrac{u_2}{u_1}\right)^{p_1-1}\!\!\!\!\!\!\!\!\!\!\!\!u_2$, a.e. in $\Omega_0$. Moreover,
$$
\partial_i\varphi_1
=\partial_i u_1-(p_1-1)\left( \dfrac{u_2}{u_1}\right)^{p_1-1}\partial_i u_2+(p_1- 1)\left( \dfrac{u_2}{u_1}\right)^{p_1}\partial_i u_1-\left( \dfrac{u_2}{u_1}\right)^{p_1-1}\partial_i u_2. 
$$
Since $\dfrac{u_2}{u_1}\leq 1$ in $\Omega_0$, we conclude that
$$
\left| \partial_i\varphi_1 \right|\leq\left( 1+(p_1- 1)\left|\dfrac{u_2}{u_1} \right|_{L_{(\Omega_0)}^{\infty}}^{p_1} \right)\left| \partial_i u_1\right|+\left( p_1\left|\dfrac{u_2}{u_1} \right|_{L_{(\Omega_0)}^{\infty}}^{p_1- 1} \right)\left| \partial_i u_2\right|.
$$

The preceding gradient estimates, together with the Sobolev chain rule applied to the corresponding Lipschitz extensions, show that $\varphi_1\in W_0^{1,\vec{\gamma}}(\Omega)$; the extensions send $(0,0)$ to zero and hence preserve the zero trace. Similarly, since $u_1\leq Cu_2$ a.e. in $\Omega$, we obtain $\varphi_2\in W_0^{1,\vec{\gamma}}(\Omega)$. This proves the claim.

Finally, set \(w_j=u_j^{p_1},~j=1,2.\) Since
\[
\left(\max\{w_1,w_2\}\right)^{1/p_1}
=\max\{u_1,u_2\}
\qquad\mbox{and}\qquad
\left(\min\{w_1,w_2\}\right)^{1/p_1}
=\min\{u_1,u_2\},
\]
we have
\[
J(\max\{w_1,w_2\})+J(\min\{w_1,w_2\})
=J(w_1)+J(w_2).
\]
Indeed, on the set $\{w_1>w_2\}$, the functions $\max\{w_1,w_2\}$ and $\min\{w_1,w_2\}$ coincide with $w_1$ and $w_2$, respectively, while on $\{w_1\leq w_2\}$ the roles are interchanged.

Set \(h=(w_1-w_2)^+.\) For $t\in(0,1)$, we have
\[
w_2+th=(1-t)w_2+t\max\{w_1,w_2\}.
\]
Since $J$ is convex,
\[
J(w_2+th)\leq(1-t)J(w_2)+tJ(\max\{w_1,w_2\}),
\]
and therefore
\[
\frac{J(w_2+th)-J(w_2)}{t}
\leq J(\max\{w_1,w_2\})-J(w_2).
\]
Since $u_1\leq Cu_2$, the Sobolev chain rule gives, along both interpolating segments, uniform bounds of the corresponding spatial derivatives and their $t$-derivatives by $C_0(|\partial_i u_1|+|\partial_i u_2|)$, for each $i$ and $0<t<1/2$. Thus, by $(a_1)$ and dominated convergence, letting $t\to0^+$, we obtain
\[
\frac{1}{p_1}
\sum_{i=1}^{N}\int_{\Omega}
a_i(|\partial_i u_2|^{p_i})
|\partial_i u_2|^{p_i-2}\partial_i u_2
\partial_i\varphi_2\,dx
\leq J(\max\{w_1,w_2\})-J(w_2).
\]

Similarly,
\[
w_1-th=(1-t)w_1+t\min\{w_1,w_2\},
\]
and hence, by the convexity of $J$,
\[
J(w_1-th)\leq(1-t)J(w_1)+tJ(\min\{w_1,w_2\}).
\]
Thus,
\[
J(w_1)-J(\min\{w_1,w_2\})
\leq\frac{J(w_1)-J(w_1-th)}{t}.
\]
Letting $t\to0^+$ gives
\[
J(w_1)-J(\min\{w_1,w_2\})
\leq
\frac{1}{p_1}
\sum_{i=1}^{N}\int_{\Omega}
a_i(|\partial_i u_1|^{p_i})
|\partial_i u_1|^{p_i-2}\partial_i u_1
\partial_i\varphi_1\,dx.
\]
Consequently,
\[
\begin{aligned}
0
&\leq
\sum_{i=1}^{N}\int_{\Omega}
a_i(|\partial_i u_1|^{p_i})
|\partial_i u_1|^{p_i-2}\partial_i u_1
\partial_i\varphi_1\,dx
-
\sum_{i=1}^{N}\int_{\Omega}
a_i(|\partial_i u_2|^{p_i})
|\partial_i u_2|^{p_i-2}\partial_i u_2
\partial_i\varphi_2\,dx.
\end{aligned}
\]

Since $u_1$ and $u_2$ are respectively a subsolution and a supersolution, and $\varphi_1,\varphi_2\geq0$, we get
\[
\begin{aligned}
0
\leq
\int_{\Omega_0}f(x,u_1)\varphi_1\,dx
-
\int_{\Omega_0}f(x,u_2)\varphi_2\,dx=
\int_{\Omega_0}
\left[
\frac{f(x,u_1)}{u_1^{p_1-1}}
-
\frac{f(x,u_2)}{u_2^{p_1-1}}
\right]
(u_1^{p_1}-u_2^{p_1})\,dx.
\end{aligned}
\]
Since $u_1>u_2$ in $\Omega_0$ and
\(
t\mapsto\frac{f(x,t)}{t^{p_1-1}}
\)
is strictly decreasing, the integrand is strictly negative a.e. in $\Omega_0$. Hence, if $|\Omega_0|>0$,
\[
\int_{\Omega_0}
\left[
\frac{f(x,u_1)}{u_1^{p_1-1}}
-
\frac{f(x,u_2)}{u_2^{p_1-1}}
\right]
(u_1^{p_1}-u_2^{p_1})\,dx<0,
\]
which is a contradiction. Therefore $|\Omega_0|=0$, and consequently \(u_1\leq u_2\quad\mbox{a.e. in }\Omega.\)
\end{proof}
\section{{Power-type nonlinearities}}

In this section we study questions about existence and multiplicity of nontrivial weak solutions to the following general class of problems
\begin{equation}\label{PLR}\tag{$P_{\lambda, r}$}
\left \{ \begin{array}{ll}
  -\displaystyle\sum_{i=1}^{N}\partial_{i}\left(a_{i}\left(\left|\partial_{i} u\right|^{p_{i}}\right)\left|\partial_{i} u\right|^{p_{i}-2}\partial_{i} u\right)=\lambda|u|^{r-2}u \; \mbox{in} \; \Omega , \; \\
  u=0, \;\mbox{on}\; \partial\Omega.
\end{array}\right.
\end{equation}
where $N\geq 1$, $\Omega \subset \real^{N}$ is a bounded smooth domain and
$\lambda\in\R$. We use the harmonic anisotropic Sobolev exponent
\[
 \overline\gamma:=N\left(\sum_{i=1}^N\frac1{\gamma_i}\right)^{-1},
 \qquad
 \gamma^*:=\begin{cases}
 \dfrac{N\overline\gamma}{N-\overline\gamma},&\overline\gamma<N,\\[3pt]
 +\infty,&\overline\gamma\geq N.
 \end{cases}
\]
For $\overline\gamma<N$, the endpoint of the continuous embedding is $\gamma_\infty:=\max\{\gamma^*,\gamma_+\}$; compactness holds below this endpoint. Unless a different range is explicitly stated, $r\in(1,\gamma^*)$. The
associated energy functional is
\begin{equation}\label{power-energy}
 I_\lambda(u)=\sum_{i=1}^N\frac1{p_i}\int_\Omega
 A_i(|\partial_i u|^{p_i})\,dx-\frac{\lambda}{r}\int_\Omega|u|^r\,dx.
\end{equation}

For later use, we set
\[
\mathcal M(u):=
\sum_{i=1}^N\|\partial_i u\|_{p_i}^{p_i}
+
\sum_{i\in\mathscr I_q}\|\partial_i u\|_{q_i}^{q_i}=\sum_{i=1}^N\|\partial_i u\|_{p_i}^{p_i}
+
\sum_{i=1}^NH(\zeta_i)\|\partial_i u\|_{q_i}^{q_i},
\qquad
u\in W_0^{1,\vec\gamma}(\Omega),
\]
where $\mathscr I_q:=\{i:H(\zeta_i)=1\}$.
\subsection{{Nonexistence results}}

\color{black}
The main purpose of the present subsection is to show that if the parameters involved in the problem belong to some suitable ranges, then the unique solution of problem \eqref{PLR} is the trivial one.

\begin{theorem}\label{teo1}
Suppose that $\lambda\in(-\infty, 0]$, $r\in (1, \gamma^{\ast})$ and there exists $j\in\{1, 2, \ldots, N\}$ such that $a_j$ satisfies $(a_1)$. Then, the unique solution of problem \eqref{PLR} is the trivial one.
\end{theorem}

\color{black}
\begin{proof}
Let $u\in W_{0}^{1, \vec{\gamma}}(\Omega)$ be a weak solution of \eqref{PLR}, that is,
$$
\displaystyle\sum_{i=1}^{N}\int_{\Omega}a_{i}\left(\left|\partial_{i} u\right|^{p_{i}}\right)\left|\partial_{i} u\right|^{p_{i}-2}\partial_{i} u\partial_{i} v dx=\lambda\int_{\Omega}|u|^{r-2}uv dx, \ \forall \ v\in W_{0}^{1, \vec{\gamma}}(\Omega).
$$
Since $\lambda\leq 0$, by choosing $v=u$, we get
$$
\displaystyle\sum_{i=1}^{N}\int_{\Omega}a_{i}\left(\left|\partial_{i} u\right|^{p_{i}}\right)\left|\partial_{i} u\right|^{p_{i}} dx=\lambda\int_{\Omega}|u|^{r} dx\leq 0.
$$
Since $a_i$ are nonnegative functions and $a_{j}$ satisfies $(a_1)$, there exists a positive number $\alpha_j$ such that
$$
0\leq\alpha_{j}\int_{\Omega}\left|\partial_{j} u\right|^{p_{j}} dx\leq\lambda\int_{\Omega}|u|^{r} dx\leq 0.
$$
The result now follows from the Poincar\'e-type anisotropic inequality \eqref{API}.
\end{proof}

The next result gives nonexistence for small positive values of $\lambda$. Set
\[
\widetilde p:=\max\bigl(\{p_N\}\cup\{q_i:H(\zeta_i)=1\}\bigr).
\]
Choose $k$ so that either $\widetilde p=p_N$ and $k=N$, or
$\widetilde p=q_k$ and $H(\zeta_k)=1$. Put $c_k=\alpha_N$ in the first
case and $c_k=\theta_k$ in the second.

\begin{theorem}
Suppose that $p_{1}<r<\widetilde p<\gamma^{\ast}$ and $(a_1)$ hold. Then, there exists $\lambda_{\ast}>0$ such that \eqref{PLR} admits only the trivial solution for $\lambda\in (0, \lambda_{\ast})$.
\end{theorem}

\color{black}
\begin{proof}
Let $u$ be a weak solution for the problem \eqref{PLR}. Thus, as in the proof of Theorem \ref{teo1}, we have
\begin{equation}\label{iden1}
\displaystyle\sum_{i=1}^{N}\int_{\Omega}a_{i}\left(\left|\partial_{i} u\right|^{p_{i}}\right)\left|\partial_{i} u\right|^{p_{i}} dx=\lambda\int_{\Omega}|u|^{r} dx.
\end{equation}
Since $p_{1}<r<\widetilde p$, interpolation gives
$$
|u|_{r}^{r}\leq |u|_{p_{1}}^{\theta r}|u|_{\widetilde p}^{(1-\theta)r},
$$
where $\theta\in(0,1)$ satisfies
$1/r=\theta/p_1+(1-\theta)/\widetilde p$. Weighted Young's inequality yields
\begin{equation}\label{ineq1}
|u|_{r}^{r}\leq
\frac{\theta r}{p_1}|u|_{p_1}^{p_1}
+
\frac{(1-\theta)r}{\widetilde p}
|u|_{\widetilde p}^{\widetilde p}.
\end{equation}
Combining \eqref{iden1}, \eqref{ineq1} and $(a_1)$ gives
\[
\alpha_1\|\partial_1u\|_{p_1}^{p_1}
+c_k\|\partial_k u\|_{\widetilde p}^{\widetilde p}
\leq
\lambda\left(
\frac{\theta r}{p_1}\|u\|_{p_1}^{p_1}
+
\frac{(1-\theta)r}{\widetilde p}
\|u\|_{\widetilde p}^{\widetilde p}
\right).
\]
By \eqref{API}, there are constants $C_1,C_2>0$, depending only on the
data, such that
\[
C_1\|u\|_{p_1}^{p_1}
+C_2\|u\|_{\widetilde p}^{\widetilde p}
\leq
\lambda\left(
\frac{\theta r}{p_1}\|u\|_{p_1}^{p_1}
+
\frac{(1-\theta)r}{\widetilde p}
\|u\|_{\widetilde p}^{\widetilde p}
\right).
\]
Choose
\[
0<\lambda<\lambda_*:=
\min\left\{
\frac{C_1p_1}{\theta r},
\frac{C_2\widetilde p}{(1-\theta)r}
\right\}.
\]
After moving the right-hand side to the left, both coefficients are
positive; hence $u=0$.
\end{proof}

\color{black}

\color{black}
\subsection{Existence results}

\color{black}
In this subsection we consider several situations involving the powers $p_i, r$ and the parameter $\lambda$. In order to prove our main theorems, the next result plays an important role.

\begin{lemma}\label{I f.s.c.i for r less than gamma}
Suppose that $r\in (1, \gamma^{\ast})$ and $(a_1)$--$(a_2)$ hold. Then the energy
functional $I_\lambda$ is weakly lower semicontinuous.
\end{lemma}
\begin{proof}
Let us denote $I_\lambda(u)=\psi(u)-\lambda\xi(u)$, where 
$$
\psi(u)=\sum_{i=1}^{N}\frac{1}{p_{i}}\int_{\Omega}A_i\left(\left|\partial_{i} u\right|^{p_{i}}\right)dx\qquad\mbox{and}\qquad \xi(u)=\frac{1}{r}\int_{\Omega}|u|^{r}dx.
$$
Since $r\in (1, \gamma^{\ast})$, it follows from compact embeddings that $\xi$ is weakly continuous. It remains to prove that $\psi$ is weakly lower semicontinuous. For this purpose, observe that defining $h_{i}:(0,\infty)\to\R$, $\ i\in\{1, 2, \ldots, N\}$, by
$$
h_{i}(t)=\frac{1}{p_i}A_{i}(t^{p_{i}}), 
$$
we have $h'_{i}(t)=a_{i}(t^{p_i})t^{p_{i}-1}$ and 
\[
h''_{i}(t)=p_i a_i'(t^{p_i})t^{2p_i-2}
+(p_i-1)a_i(t^{p_i})t^{p_i-2}>0,
\qquad t>0,
\]
where $(a_2)$ and the positivity of $a_i$ were used. By continuity at zero, the derivative extends there with value zero. Thus $t\mapsto A_i(|t|^{p_i})/p_i$ is convex on $\R$, and $\psi$ is convex. Now, we shall prove that $\psi$ is continuous. In fact, if
$u_n\to u$ in $W_0^{1,\vec{\gamma}}(\Omega)$, then, up to a
subsequence,
\[
\left|\partial_i u_n\right|
\to
\left|\partial_i u\right|
\qquad\mbox{and}\qquad
\left|\partial_i u_n\right|
\leq g_i
\quad\mbox{a.e. in }\Omega,
\]
for some $g_i\in L^{\gamma_i}(\Omega)$ and each
$i\in\{1,2,\ldots,N\}$.

Since $h_i$ is continuous and increasing, we conclude that
\begin{equation}\label{vain1}
\sum_{i=1}^{N}
h_i\left(
\left|\partial_i u_n\right|
\right)
\to
\sum_{i=1}^{N}
h_i\left(
\left|\partial_i u\right|
\right)
\quad\mbox{and}\qquad
\sum_{i=1}^{N}
h_i\left(
\left|\partial_i u_n\right|
\right)
\leq
\sum_{i=1}^{N}h_i(g_i)
\quad\mbox{a.e. in }\Omega.
\end{equation}

It follows from \eqref{desig of item a1} that
\[
0\leq h_i(g_i)
\leq
\frac{\beta_i}{p_i}g_i^{p_i}
+
H(\zeta_i)\frac{\theta_i}{q_i}g_i^{q_i}.
\]
Since $\gamma_i=p_i$ if $H(\zeta_i)=0$ and
$\gamma_i=q_i$ if $H(\zeta_i)=1$, we conclude that
$h_i(g_i)\in L^1(\Omega)$. Therefore, from \eqref{vain1} and
the Lebesgue Dominated Convergence Theorem, we obtain
\[
\psi(u_n)\to\psi(u).
\]
Thus, $\psi$ is continuous and, by convexity, weakly lower
semicontinuous.
\end{proof}

\begin{theorem}\label{Theorem p1 smallest r smallest critical range}
Suppose that $(a_1)-(a_2)$ hold for all
$i\in\{1, 2, \ldots, N\}$ and
$1<p_{1}<r<\widetilde p<\gamma^{\ast}$. Then there exists
$\lambda^{*}>0$ such that problem \eqref{PLR} possesses a nontrivial
solution for every $\lambda>\lambda^{*}$.
\end{theorem}
\begin{proof}
Write $I=I_\lambda$ as in \eqref{power-energy}. By \eqref{API} and H\"older's inequality,
$\|u\|_r^r\leq C\|\partial_k u\|_{\widetilde p}^r$.
Since $r<\widetilde p$, Young's inequality absorbs this reaction term
into any fixed positive fraction of its $\widetilde p$-energy term.
Using \eqref{energy-growth}, we obtain
\[
I_\lambda(u)\geq c\mathcal M(u)-C_\lambda,
\qquad c>0.
\]
There are finitely many derivative norms in \eqref{norm 2 sobolev space},
and each power is greater than one. Hence $\mathcal M(u)\to\infty$
whenever $\|u\|\to\infty$. Thus
$I_\lambda$ is coercive. Lemma \ref{I f.s.c.i for r less than gamma} and
the direct method give a global minimizer $u_\lambda$.

Fix $0\ne v\in C_c^\infty(\Omega)$. From $(a_1)$,
\[
 I_\lambda(v)\leq \sum_{i=1}^N\frac{\beta_i}{p_i}\|\partial_i v\|_{p_i}^{p_i}
+
\sum_{i=1}^N
H(\zeta_i)\frac{\theta_i}{q_i}
\|\partial_i v\|_{q_i}^{q_i}-\frac{\lambda}{r}\|v\|_r^r.
\]

The right-hand side is negative for all sufficiently large $\lambda$.
Consequently $I_\lambda(u_\lambda)\leq I_\lambda(v)<0=I_\lambda(0)$, so
$u_\lambda\ne0$. Replacing $u_\lambda$ by $|u_\lambda|$ if necessary and
using the Euler equation concludes the proof.
\end{proof}

\begin{theorem}\label{Teorema p1 menor igual pN menor igual gama critico}
Suppose that $(a_1)-(a_2)$ hold, for all $i\in\{1,2,...,N\}$ and $1<r<p_1$. Then the problem $(P_{\lambda,r})$ has at least one non-trivial and non-negative solution for each $\lambda\in (0,+\infty)$. 

Moreover, if $u_1$ and $u_2$ are two nonnegative solutions such that, for some constants $C_1,C_2>0$,
$u_1\leq C_1u_2,~u_2\leq C_2u_1~\mbox{a.e. in }\Omega,$ then $u_1=u_2$.
\end{theorem}
\begin{proof}
Since $\gamma^*>p_1$ when finite, and otherwise $\gamma^*=+\infty$, the assumption $r<p_1$ ensures the required subcritical embedding. By $(a_1)$ and \eqref{API}, for every $i$ there is $C_i>0$ such
that $\|u\|_r^r\leq C_i\|\partial_i u\|_{p_i}^r$. Since
$r<p_1\leq p_i<q_i$, estimate \eqref{desig of item a1} shows that
$I_\lambda$ is coercive. Lemma \ref{I f.s.c.i for r less than gamma} and
the direct method therefore give a global minimizer $u_*$. For any
$0\ne v\in C_c^\infty(\Omega)$,
\[
I_\lambda(tv)\leq
\sum_i\frac{\beta_i}{p_i}t^{p_i}
\|\partial_i v\|_{p_i}^{p_i}
+
\sum_i H(\zeta_i)\frac{\theta_i}{q_i}t^{q_i}
\|\partial_i v\|_{q_i}^{q_i}
-
\frac{\lambda}{r}t^r\|v\|_r^r.
\]
Because $r<p_1\leq p_i<q_i$, the right-hand side is negative for all
sufficiently small $t>0$. Hence $I_\lambda(u_*)<0$ and $u_*\ne0$.
Replacing $u_*$ by $|u_*|$ gives a nonnegative minimizer and the Euler
equation gives a weak solution.

Finally, since $r<p_1$, the map $t\mapsto\lambda t^{r-p_1}$ is strictly decreasing.
Hence, applying Theorem \ref{Comparison principle} first to $u_1,u_2$ and then to $u_2,u_1$, we obtain
$u_1\leq u_2$ and $u_2\leq u_1$ a.e. in $\Omega$. Therefore, $u_1=u_2$.

\end{proof}

\begin{remark}
It is worth mentioning that the positivity of the non-trivial
nonnegative solution obtained above does not follow directly from a
classical strong maximum principle. In fact, in the anisotropic case,
the strong maximum principle may have a different behavior from the
isotropic one; see \cite{VetoisSMP}. This is one reason why the
comparison principle was formulated for nonnegative sub- and
supersolutions. In particular, the uniqueness assertion above does
not require the solutions to be positive in $\Omega$.
\end{remark}

\begin{remark}
If $u_1$ and $u_2$ are two distinct nonnegative solutions, then
at least one of the inequalities \( u_1\leq Cu_2~\mbox{or}~ u_2\leq Cu_1\)
fails a.e. in $\Omega$ for every $C>0$. In particular, if both
solutions are positive a.e. in $\Omega$, then
\[
\operatorname*{ess\,sup}_{\Omega}\frac{u_1}{u_2}=+\infty
\qquad\mbox{or}\qquad
\operatorname*{ess\,sup}_{\Omega}\frac{u_2}{u_1}=+\infty.
\]
\end{remark}

We next establish the mountain-pass geometry throughout the intermediate
range. The estimate retains every directional factor in the Sobolev
inequality and measures the principal energy through its full modular.
This avoids additional restrictions on the smallest effective exponent.

\begin{lemma}\label{lema r less critical and p tilde consider the cases}
Assume $(a_1)$--$(a_2)$ for every $i$, and let $p_1<r<\gamma^*$.
For every fixed $\lambda>0$, zero is a strict local minimum of
$I_\lambda$. More precisely, there exist $c,C,\delta>0$, independent
of $u$, such that
\begin{equation}\label{full-modular-local-bound}
 I_\lambda(u)\geq c\mathcal M(u)
       -\lambda C\mathcal M(u)^{1+\delta},
 \qquad u\in W_0^{1,\vec\gamma}(\Omega).
\end{equation}
Consequently, there exist $\rho,b>0$ such that
$I_\lambda(u)\geq b$ whenever $\|u\|=\rho$, with the norm
\eqref{norm 2 sobolev space}. No condition of the form
$\gamma_-<r$ or $\gamma_-<\gamma^*/N$ is required.
\end{lemma}
\begin{proof}
The lower bound in \eqref{desig of item a1} gives
\begin{equation}\label{local-energy-modular}
 I_\lambda(u)\geq c_0\mathcal M(u)-\frac{\lambda}{r}\|u\|_r^r,
 \qquad c_0>0.
\end{equation}
For every $i$, the term
$\|\partial_i u\|_{\gamma_i}^{\gamma_i}$ occurs in $\mathcal M(u)$.
Thus
\begin{equation}\label{local-directional-modular}
 \|\partial_i u\|_{\gamma_i}\leq\mathcal M(u)^{1/\gamma_i}.
\end{equation}
The lower $p_1$-phase is also retained, even if the $q_1$-phase is
active. By \eqref{API},
\begin{equation}\label{local-baseline-poincare}
 \|u\|_{p_1}\leq C\|\partial_1u\|_{p_1}
                  \leq C\mathcal M(u)^{1/p_1}.
\end{equation}

Suppose first that $\overline\gamma<N$. Theorem
\ref{anisotropic inequality} and \eqref{local-directional-modular} imply
\begin{equation}\label{local-troisi-modular}
 \|u\|_{\gamma^*}
 \leq C_T\prod_{i=1}^N\|\partial_i u\|_{\gamma_i}^{1/N}
 \leq C_T\mathcal M(u)^{1/\overline\gamma}.
\end{equation}
Since $p_1<r<\gamma^*$, there is a unique $\eta\in(0,1)$ with
\[
 \frac1r=\frac\eta{\gamma^*}+\frac{1-\eta}{p_1}.
\]
H\"older interpolation, \eqref{local-baseline-poincare}, and
\eqref{local-troisi-modular} give
\[
 \|u\|_r^r\leq
 \|u\|_{\gamma^*}^{\eta r}\|u\|_{p_1}^{(1-\eta)r}
 \leq C\mathcal M(u)^{\eta r/\overline\gamma+(1-\eta)r/p_1}.
\]
The exponent is strictly greater than one because
\[
 \frac{\eta r}{\overline\gamma}+
 \frac{(1-\eta)r}{p_1}
 =1+\eta r\left(\frac1{\overline\gamma}-\frac1{\gamma^*}\right)
 =1+\frac{\eta r}{N}.
\]
Take $\delta=\eta r/N>0$.

If $\overline\gamma\geq N$, fix a finite exponent
$t>\max\{r,\overline\gamma\}$. The finite-exponent product inequality
in Theorem \ref{anisotropic inequality} yields
\[
 \|u\|_t\leq C_t\prod_{i=1}^N\|\partial_i u\|_{\gamma_i}^{1/N}
          \leq C_t\mathcal M(u)^{1/\overline\gamma}.
\]
Choose $\eta\in(0,1)$ with
$1/r=\eta/t+(1-\eta)/p_1$. Interpolating with
\eqref{local-baseline-poincare}, we obtain
\[
 \|u\|_r^r\leq C\mathcal M(u)^{1+\delta},
 \qquad
 \delta=\eta r\left(\frac1{\overline\gamma}-\frac1t\right)>0.
\]
All constants are fixed after $t$ has been chosen. In both cases,
substitution in \eqref{local-energy-modular} proves
\eqref{full-modular-local-bound}.

To obtain the uniform barrier, put $L=N+|\mathscr I_q|$ and
$\kappa=\widetilde p=\gamma_+$. If $0<\|u\|\leq1$, every derivative
norm in \eqref{norm 2 sobolev space} is at most $\|u\|$, and at least
one is at least $\|u\|/L$. Since all its associated exponents lie
between $p_1$ and $\kappa$, we have
\begin{equation}\label{local-modular-norm-bounds}
 (\|u\|/L)^\kappa\leq\mathcal M(u)
                       \leq L\|u\|^{p_1}.
\end{equation}
Choose $0<\rho\leq1$ so small that
$\lambda C(L\rho^{p_1})^\delta\leq c_0/2$.
Then, for $0<\|u\|\leq\rho$,
\[
 I_\lambda(u)\geq\frac{c_0}{2}\mathcal M(u)>0.
\]
On $\|u\|=\rho$, \eqref{local-modular-norm-bounds} gives the
positive barrier $b=(c_0/2)(\rho/L)^\kappa$.
\end{proof}

The next theorem combines this unrestricted local estimate with the
negative global minimum and the mountain-pass theorem; see \cite{AM}.

\begin{theorem}\label{palais smale r menor p tilde}
Assume $(a_1)$--$(a_2)$ for every $i$ and
$1<p_1<r<\widetilde p<\gamma^*$. There is $\lambda_0>0$ such that,
for every $\lambda>\lambda_0$, problem \eqref{PLR} has two distinct
nontrivial weak solutions $u_\lambda,v_\lambda$ satisfying
\[
 I_\lambda(u_\lambda)<0<I_\lambda(v_\lambda).
\]
The solution $u_\lambda$ can be chosen nonnegative and is a global
minimizer. There are no further restrictions on $\gamma_-$ or
$\overline\gamma$.
\end{theorem}
\begin{proof}
The proof of Theorem
\ref{Theorem p1 smallest r smallest critical range} gives, for every
sufficiently large $\lambda$, a nonnegative global minimizer
$u_\lambda\neq0$ with $I_\lambda(u_\lambda)<0$. Lemma
\ref{lema r less critical and p tilde consider the cases} gives
$\rho,b>0$ with $I_\lambda\geq b$ on $\|u\|=\rho$, and
$I_\lambda(u)>0$ for $0<\|u\|\leq\rho$. In particular,
$\|u_\lambda\|>\rho$.

We verify compactness at every energy level. Coercivity bounds every
Palais--Smale sequence $(u_n)$. Passing to a subsequence, let
$u_n\rightharpoonup u$ in $W_0^{1,\vec\gamma}(\Omega)$.
Since $r<\gamma^*$, Theorem \ref{FraGaKow04} gives
$u_n\to u$ in $L^r(\Omega)$. H\"older's inequality implies
\[
 \left|\int_\Omega |u_n|^{r-2}u_n(u_n-u)\,dx\right|
 \leq\|u_n\|_r^{r-1}\|u_n-u\|_r\longrightarrow0.
\]
Also $I_\lambda'(u_n)[u_n-u]\to0$, because the sequence of
differences is bounded. Therefore
\[
 \langle\mathcal A(u_n),u_n-u\rangle\longrightarrow0.
\]
Proposition \ref{Splus} yields strong convergence in the energy space.
Thus $I_\lambda$ satisfies the Palais--Smale condition.

Let $\Gamma$ be the family of continuous paths from $0$ to
$u_\lambda$ in this space and set
\[
 c_\lambda=\inf_{g\in\Gamma}\max_{t\in[0,1]}I_\lambda(g(t)).
\]
Every such path crosses $\|u\|=\rho$, so $c_\lambda\geq b>0$.
The mountain-pass theorem gives a critical point $v_\lambda$ with
$I_\lambda(v_\lambda)=c_\lambda$. Its positive energy distinguishes
it from both zero and $u_\lambda$, proving the assertion.
\end{proof}

\subsection{{Partial resonance: \texorpdfstring{$r=p_1<p_N$}{r=p1<pN}}}

The equality $r=p_1$ leads to a different local geometry when the
smallest directional exponent does not occur in every direction. We
now determine the exact set of parameters admitting nontrivial
solutions and prove infinite multiplicity for each parameter in that
set. Only $(a_1)$--$(a_2)$ are needed; continuity of $a_i$ at zero
replaces any differentiability assumption there.

Throughout this subsection put $p=p_1<p_N$ and
\[
 Y=\{i:p_i=p\},\qquad Z=\{i:p_i>p\},\qquad
 X=W_0^{1,\vec\gamma}(\Omega).
\]
Both index sets are nonempty, and in particular $N\geq2$. Since
$\gamma_i\geq p_i\geq p$, one has $p<\gamma^*$, whether the
endpoint is finite or infinite. Thus $X$ embeds compactly in
$L^p(\Omega)$. Define the partial Rayleigh threshold
\begin{equation}\label{partial-rayleigh-threshold}
 \Lambda_Y(\Omega)=
 \inf_{0\neq v\in C_c^\infty(\Omega)}
 \frac{\displaystyle\sum_{i\in Y}a_i(0)
                \int_\Omega|\partial_i v|^p\,dx}
      {\displaystyle\int_\Omega|v|^p\,dx}.
\end{equation}
Here the derivatives in $Z$ are deliberately absent: they contribute
only at orders strictly higher than $p$ at small amplitudes.

\begin{lemma}\label{partial-threshold-test-space}
The number $\Lambda_Y(\Omega)$ belongs to $(0,\infty)$, and
\begin{equation}\label{partial-poincare}
 \sum_{i\in Y}a_i(0)\|\partial_i u\|_p^p
                 \geq\Lambda_Y(\Omega)\|u\|_p^p,
 \qquad u\in X.
\end{equation}
For every $\lambda>\Lambda_Y(\Omega)$ there are bounded open sets
$D\subset\R^{|Y|}$, $E\subset\R^{|Z|}$, with
$\overline{D\times E}\subset\Omega$, and a nonzero
$\varphi\in C_c^\infty(D)$ such that
\begin{equation}\label{partial-section-rayleigh}
 R_Y(\varphi):=
 \frac{\displaystyle\sum_{i\in Y}a_i(0)
                         \int_D|\partial_i\varphi|^p\,dy}
      {\displaystyle\int_D|\varphi|^p\,dy}<\lambda.
\end{equation}
Coordinates in $Y$ and $Z$ are denoted by $y$ and $z$, respectively;
this is a regrouping of the fixed coordinate directions.
\end{lemma}
\begin{proof}
Fix $i\in Y$. Directional Poincar\'e gives
$\|v\|_p\leq C_i\|\partial_i v\|_p$, so the quotient in
\eqref{partial-rayleigh-threshold} is at least
$a_i(0)C_i^{-p}>0$. Any fixed nonzero smooth compactly supported
function gives a finite upper bound.

We recall the density needed to pass this inequality to $X$.
Directional Poincar\'e implies $u\in L^{\gamma_i}(\Omega)$ for
every $i$. The zero extension of a function in $X$ has distributional
derivatives equal to the zero extensions of $\partial_i u$, since
its $W^{1,1}$ trace vanishes. On a smooth bounded domain, a finite
boundary partition of unity, small translations into the domain in
each boundary chart, and convolution produce approximations in
$C_c^\infty(\Omega)$. Translations and convolution preserve the
fixed derivative directions and converge in each $L^{\gamma_i}$;
the terms arising from the partition are controlled by
$u\in L^{\gamma_i}$. The approximation therefore converges in $X$.
Because $\gamma_i\geq p$ in $Y$, its derivatives in these directions
also converge in $L^p$, and its functions converge in $L^p$.
Passing the defining inequality to the limit proves
\eqref{partial-poincare}.

Choose $0\neq v\in C_c^\infty(\Omega)$ with partial quotient less
than $\lambda$, and extend it by zero to $\R^N$. Define
\[
 B(z)=\int_{\R^{|Y|}}|v(y,z)|^p\,dy,
 \qquad
 Q(z)=\sum_{i\in Y}a_i(0)
            \int_{\R^{|Y|}}|\partial_i v(y,z)|^p\,dy.
\]
Fubini's theorem gives $\int Q(z)\,dz<\lambda\int B(z)\,dz$.
Consequently some $z_0$ satisfies $B(z_0)>0$ and
$Q(z_0)<\lambda B(z_0)$. Put $\varphi(y)=v(y,z_0)$.
The set $\operatorname{supp}\varphi\times\{z_0\}$ is compactly
contained in $\Omega$. Taking sufficiently small neighborhoods of
its factors gives bounded open $D,E$ with
$\overline{D\times E}\subset\Omega$ and
$\operatorname{supp}\varphi\subset D$. This proves
\eqref{partial-section-rayleigh}. Notice that no product structure
of the whole domain has been assumed.
\end{proof}

We will use the following standard Clark--Heinz principle. If an even
functional $J\in C^1(X,\R)$ is bounded below, satisfies the
Palais--Smale condition, has $J(0)=0$, and, for every $m\in\N$,
there is an $m$-dimensional subspace $V_m$ and a radius $\rho_m>0$
such that
\[
 \sup\{J(v):v\in V_m,\ \|v\|=\rho_m\}<0,
\]
then $J$ has negative critical values $c_m$ tending to zero. This is
the genus minimax result stated as Theorem A in \cite{LiuWangClark}.
We verify every hypothesis below and identify zero as the only
critical point at energy zero to obtain strong convergence of the
solutions themselves.

\begin{theorem}\label{partial-resonance-infinite}
Assume $(a_1)$--$(a_2)$ for every $i$, and let
$1<r=p_1=:p<p_N$. Let $\Lambda_Y(\Omega)$ be defined by
\eqref{partial-rayleigh-threshold}. Then:
\begin{enumerate}[label=\textup{(\roman*)}]
\item For every $\lambda\leq\Lambda_Y(\Omega)$, zero is the only
weak solution of \eqref{PLR}.
\item For each fixed $\lambda>\Lambda_Y(\Omega)$, there is a
sequence of nonzero weak solutions $u_n\in X$ such that
$u_n\neq\pm u_m$ for $n\neq m$, and
\[
 I_\lambda(u_n)<0,\qquad I_\lambda(u_n)\longrightarrow0,
 \qquad \|u_n\|\longrightarrow0.
\]
In particular, there are infinitely many distinct pairs
$\{u_n,-u_n\}$ for that same parameter.
\item For $\lambda>\Lambda_Y(\Omega)$, the functional has a
nonzero nonnegative global minimizer of negative energy, whereas
zero is not a local minimum. Every nonzero critical point has
strictly negative energy.
\end{enumerate}
No assumption $(a_3)$, no upper restriction on $\widetilde p$ relative
to $\gamma^*$, and no assertion about the sign of the sequence in
\textup{(ii)} are required.
\end{theorem}
\begin{proof}
\emph{Step 1: the necessary threshold.}
Let $u\in X$ be a weak solution. Since $a_i$ is nondecreasing,
$a_i(s)\geq a_i(0)>0$ for all $s\geq0$. Testing with $u$ and using
\eqref{partial-poincare}, we obtain
\begin{align*}
 \lambda\|u\|_p^p
 &=\sum_i\int_\Omega
       a_i(|\partial_i u|^{p_i})|\partial_i u|^{p_i}\,dx\\
 &\geq\Lambda_Y(\Omega)\|u\|_p^p
       +\sum_{i\in Z}\int_\Omega
       a_i(|\partial_i u|^{p_i})|\partial_i u|^{p_i}\,dx.
\end{align*}
If $u\neq0$, fix any $j\in Z$. Inequality \eqref{API} with exponent
$p_j$ gives $\|u\|_{p_j}\leq C_j\|\partial_j u\|_{p_j}$;
thus this derivative is nonzero and the last sum is strictly
positive. It follows that $\lambda>\Lambda_Y(\Omega)$, proving (i).

\emph{Step 2: coercivity and compactness.}
The principal functional is $C^1$ by Proposition \ref{Splus}, and
the compact embedding $X\hookrightarrow L^p(\Omega)$ makes the
reaction functional $C^1$. The energy is even and $I_\lambda(0)=0$.
By \eqref{desig of item a1},
\[
 I_\lambda(u)\geq c_0\mathcal M(u)-\frac\lambda p\|u\|_p^p,
 \qquad c_0>0.
\]
Choose $j\in Z$. H\"older and directional Poincar\'e give
$\|u\|_p^p\leq C\|\partial_j u\|_{p_j}^p$.
Since $p<p_j$, Young's inequality yields
\[
 \frac\lambda p\|u\|_p^p
 \leq\frac{c_0}{2}\|\partial_j u\|_{p_j}^{p_j}+C_\lambda.
\]
Hence $I_\lambda(u)\geq(c_0/2)\mathcal M(u)-C_\lambda$.
As $\|u\|\to\infty$, the maximum of the finitely many derivative
norms tends to infinity, so $\mathcal M(u)\to\infty$. The energy is
therefore coercive and bounded below.

For a Palais--Smale sequence, coercivity gives boundedness in $X$.
After extraction, $u_n\rightharpoonup u$ in $X$ and
$u_n\to u$ in $L^p(\Omega)$. Therefore
\[
 \left|\int_\Omega |u_n|^{p-2}u_n(u_n-u)\,dx\right|
 \leq\|u_n\|_p^{p-1}\|u_n-u\|_p\longrightarrow0.
\]
The derivative identity gives
$\langle\mathcal A(u_n),u_n-u\rangle\to0$.
Proposition \ref{Splus} implies $u_n\to u$ in $X$.
This proves Palais--Smale at every level.

\emph{Step 3: negative spheres of every finite dimension for one
fixed parameter.}
Fix $\lambda>\Lambda_Y(\Omega)$ and take $D,E,\varphi$ from Lemma
\ref{partial-threshold-test-space}. Write
$d=\lambda-R_Y(\varphi)>0$. For arbitrary $m\in\N$, choose an
$m$-dimensional subspace $F_m\subset C_c^\infty(E)$ and put
\[
 V_m=\{\varphi(y)\psi(z):\psi\in F_m\},
\]
extending its elements by zero outside $D\times E$. These elements
belong to $C_c^\infty(\Omega)$, and multiplication by
$\varphi\neq0$ is injective, so $\dim V_m=m$.
For $w=\varphi\psi$, separation of the integrals gives
\begin{equation}\label{partial-negative-leading-form}
 \sum_{i\in Y}a_i(0)\|\partial_i w\|_p^p
       -\lambda\|w\|_p^p=-d\|w\|_p^p.
\end{equation}

Let $S_m=\{w\in V_m:\|w\|=1\}$. This set is compact and consists
of smooth functions supported in a fixed compact subset of
$\Omega$. Equivalence of norms on $V_m$ gives constants $K_m<\infty$
and $b_m>0$ such that
\[
 \max_i\|\partial_i w\|_\infty\leq K_m,
 \qquad \|w\|_p^p\geq b_m,\qquad w\in S_m.
\]
Set
$\omega_i(h)=\sup_{0\leq s\leq h}|a_i(s)-a_i(0)|$.
Continuity at zero gives $\omega_i(h)\to0$ as $h\downarrow0$, and
\[
 |A_i(s)-a_i(0)s|\leq s\omega_i(s).
\]
For $i\in Y$, consequently,
\[
 \sup_{w\in S_m}\left|
 t^{-p}\int_\Omega A_i(t^p|\partial_i w|^p)\,dx
       -a_i(0)\|\partial_i w\|_p^p\right|
 \leq |\Omega|K_m^p\omega_i(t^pK_m^p)\longrightarrow0.
\]
For $i\in Z$, $a_i$ is bounded near zero, so for sufficiently small
$t$ and all $w\in S_m$,
\[
 0\leq t^{-p}\int_\Omega A_i(t^{p_i}|\partial_i w|^{p_i})\,dx
       \leq C_m t^{p_i-p}\longrightarrow0.
\]
Combining these uniform estimates with
\eqref{partial-negative-leading-form}, we obtain
\[
 t^{-p}I_\lambda(tw)=-\frac d p\|w\|_p^p+o_m(1),
 \qquad w\in S_m,
\]
where $o_m(1)\to0$ uniformly as $t\downarrow0$.
Choose $\rho_m>0$ so small that the absolute value of this remainder
is at most $db_m/(2p)$. Then
\[
 \sup_{w\in S_m}I_\lambda(\rho_m w)
       \leq-\frac{db_m}{2p}\rho_m^p<0.
\]
The parameter $\lambda$ has remained fixed throughout; only
$V_m$ and $\rho_m$ depend on $m$. The Clark--Heinz principle now
gives negative critical values converging to zero.

\emph{Step 4: the energy identity and strong convergence to zero.}
If $I_\lambda'(u)=0$, put $s_i=|\partial_i u|^{p_i}$ and subtract
$I_\lambda'(u)[u]/p$ from the energy. Monotonicity implies
$A_i(s)\leq s a_i(s)$, so
\begin{align}
 I_\lambda(u)
 &=\sum_i\int_\Omega
       \left(\frac{A_i(s_i)}{p_i}-\frac{a_i(s_i)s_i}{p}\right)dx
       \notag\\
 &\leq\sum_{i\in Z}\left(\frac1{p_i}-\frac1p\right)
                         \int_\Omega a_i(s_i)s_i\,dx.
 \label{partial-critical-energy}
\end{align}
For $u\neq0$ the last expression is strictly negative, by the same
directional Poincar\'e argument as in Step 1. Thus zero is the only
critical point with energy zero. Select distinct negative critical
values tending to zero and corresponding critical points $u_n$.
Distinct values ensure $u_n\neq\pm u_m$. The sequence is
Palais--Smale at zero; every subsequence has a further subsequence
converging strongly to a critical point of energy zero, hence to
zero. It follows that the whole selected sequence converges to zero
in $X$, proving (ii).

\emph{Step 5: the global minimum and the geometry at zero.}
Coercivity and Lemma \ref{I f.s.c.i for r less than gamma} give a
global minimizer. Step 3 with $m=1$ gives negative energy, so the
minimizer is nonzero; replacing it by its absolute value preserves
the energy and yields a nonnegative minimizer. Since the energy is
$C^1$, this minimizer solves the full weak equation. Moreover Step 3
provides negative values at arbitrarily small nonzero amplitudes,
which proves that zero is not a local minimum. Identity
\eqref{partial-critical-energy} proves the remaining statement in
(iii).
\end{proof}

\begin{remark}\label{partial-contribution}
The conclusion determines the exact admissible parameter set
$(\Lambda_Y(\Omega),\infty)$ and gives infinitely many pairs for
each fixed parameter in that interval, together with strong
convergence to zero. In the homogeneous anisotropic case,
\cite[Theorem 1.2]{FigueiredoSantosSuarez2015} gives arbitrarily many
pairs for sufficiently large parameters, with a threshold depending
on the prescribed number of pairs. The present conclusion at
$r=p_1<p_N$ identifies a single exact threshold and remains valid
for the nonhomogeneous coefficients satisfying $(a_1)$--$(a_2)$.
The threshold depends on the small-amplitude coefficients $a_i(0)$
in $Y$ and the domain, while all active higher phases are permitted.
Neither positivity nor nodality of the sequence is asserted.
\end{remark}

\begin{corollary}\label{partial-disk-example}
Let $\Omega=B_R(0)\subset\R^2$, $P>2$, and consider
\[
 -\partial_{11}u
 -\partial_2\bigl(|\partial_2u|^{P-2}\partial_2u\bigr)
 =\lambda u\quad\hbox{in }B_R(0),\qquad u=0\quad\hbox{on }\partial B_R(0).
\]
Its partial threshold is $\Lambda_Y=\pi^2/(4R^2)$. For each
$\lambda>\pi^2/(4R^2)$ it has infinitely many distinct pairs of
nonzero weak solutions converging to zero in
$W_0^{1,(2,P)}(B_R(0))$. For $\lambda\leq\pi^2/(4R^2)$ it has only
zero. In contrast, $-\Delta u=\lambda u$ on the same disk has only
zero throughout
\[
 \frac{\pi^2}{4R^2}<\lambda<\frac{\pi^2}{2R^2}.
\]
\end{corollary}
\begin{proof}
Take $p_1=2$, $p_2=P$ and $a_1=a_2=1$, with inactive higher phases.
These choices satisfy $(a_1)$--$(a_2)$. For each fixed $x_2$, the
horizontal section of the disk is an interval of length at most
$2R$. The sharp one-dimensional Dirichlet Poincar\'e inequality,
integrated over $x_2$, gives
\[
 \|\partial_1v\|_2^2\geq\frac{\pi^2}{4R^2}\|v\|_2^2,
 \qquad v\in C_c^\infty(B_R(0)).
\]
Thus $\Lambda_Y\geq\pi^2/(4R^2)$. For $0<\ell<R$, choose $h>0$
with $\ell^2+h^2<R^2$ and take product test functions supported in
$(-\ell,\ell)\times(-h,h)$. Approximate the first Dirichlet
eigenfunction on $(-\ell,\ell)$ by functions in
$C_c^\infty(-\ell,\ell)$ in $H_0^1$, and multiply by a fixed
nonzero element of $C_c^\infty(-h,h)$. Their partial quotients tend
to $\pi^2/(4\ell^2)$. Letting $\ell\uparrow R$ proves equality.
Theorem \ref{partial-resonance-infinite} now applies.

For the Laplacian, the same sectionwise argument in each coordinate
direction gives
\[
 \int_{B_R}|\nabla u|^2\,dx
 \geq\frac{\pi^2}{2R^2}\int_{B_R}|u|^2\,dx,
 \qquad u\in H_0^1(B_R).
\]
Testing the linear weak equation with $u$ excludes every nonzero
solution for $\lambda<\pi^2/(2R^2)$, proving the comparison.
\end{proof}

The mechanism behind the preceding comparison can already be seen from the behavior of the energy near the origin. The directions in $Y$ contribute terms of order $p$, while the directions in $Z$ contribute only terms of strictly higher order, since $p_i>p$ for every $i\in Z$. Consequently, on the product spaces constructed above, the leading-order part of the energy depends only on the $Y$-derivatives and on the reaction term. The variables associated with $Z$ may therefore vary in subspaces of arbitrarily large finite dimension without producing any derivative contribution at order $p$. In the quadratic example, this means that the leading-order energy near zero is determined by
\[
\int_\Omega |\partial_1u|^2\,dx-\lambda\int_\Omega |u|^2\,dx,
\]
whereas the term involving $\partial_2u$ has order $P>2$. Hence, once $\lambda$ exceeds the partial threshold, the leading-order energy admits negative subspaces of arbitrarily large finite dimension. At the same time, the higher-order directional term remains present in the full nonlinear energy and provides the coercivity and compactness needed at Palais--Smale levels. This is precisely what distinguishes the anisotropic problem from the Laplacian: in the latter, both coordinate derivatives already appear at quadratic order, so no direction is absent from the leading-order energy.

\subsection{{Singular case \texorpdfstring{$0<r<1$}{0<r<1}}}

The proof of the result below is motivated by the variational approach
used in \cite{Yijing} for singular semilinear elliptic problems. Here,
we adapt these ideas to the present anisotropic nonhomogeneous setting.

\begin{theorem}\label{Singular Case Theorem}
Assume that $(a_1)-(a_2)$ and $0<r<1<p_1$ hold.
Then the problem \eqref{PLR} has at least one non-trivial solution
$u$ such that $u>0$ a.e. in $\Omega$ for each $\lambda\in(0,+\infty)$. 

Moreover, this positive a.e. weak solution is unique; no boundedness assumption on quotients is required.
\end{theorem}
\begin{proof}
Let
$$
 \Psi(u)=\sum_{i=1}^N\frac1{p_i}\int_\Omega
 A_i(|\partial_i u|^{p_i})\,dx,
 \qquad
 \mathcal I_\lambda(u)=\Psi(u)-\frac\lambda r
 \int_\Omega (u^+)^r\,dx.
$$
Since $0<r<1<p_1\leq p_i$, \eqref{API}, H\"older's inequality and
$(a_1)$ show that $\mathcal I_\lambda$ is coercive. If $u_n\rightharpoonup u$
in $W_0^{1,\vec\gamma}(\Omega)$, then $u_n\to u$ in $L^1(\Omega)$ and
$|(s^+)^r-(t^+)^r|\leq|s-t|^r$; hence the reaction term is continuous
along the sequence. The convex principal part is weakly lower
semicontinuous, so $\mathcal I_\lambda$ has a global minimizer.
Since
\[
\mathcal I_\lambda(u^+)\leq\mathcal I_\lambda(u),
\]
we may choose a nonnegative global minimizer $u$.

For every fixed $0\ne v\geq0$ in $C_c^\infty(\Omega)$, by
\eqref{desig of item a1}, for $0<t\leq1$,
$$
\begin{aligned}
\mathcal I_\lambda(tv)
\leq
\sum_{i=1}^N\frac{\beta_i}{p_i}t^{p_i}
\|\partial_i v\|_{p_i}^{p_i}
+
\sum_{i=1}^N
H(\zeta_i)\frac{\theta_i}{q_i}t^{q_i}
\|\partial_i v\|_{q_i}^{q_i}
-\frac{\lambda}{r}t^r\|v\|_r^r
\leq
C(v)t^{p_1}
-\frac{\lambda}{r}t^r\|v\|_r^r,
\end{aligned}
$$
where
\[
C(v):=
\sum_{i=1}^N\frac{\beta_i}{p_i}
\|\partial_i v\|_{p_i}^{p_i}
+
\sum_{i=1}^N
H(\zeta_i)\frac{\theta_i}{q_i}
\|\partial_i v\|_{q_i}^{q_i}.
\]
Since $0<r<p_1$, we have
$\mathcal I_\lambda(tv)<0$ for all sufficiently small $t>0$.
Thus $u\ne0$.

For $0\leq\varphi\in W_0^{1,\vec\gamma}(\Omega)$ and $t>0$, minimality gives
\[
\frac\lambda r\int_\Omega\frac{(u+t\varphi)^r-u^r}{t}\,dx
\leq\frac{\Psi(u+t\varphi)-\Psi(u)}{t}.
\]
The integrand is nonnegative. Fatou's lemma and the differentiability of $\Psi$ (cf. \cite{Yijing}) yield, for every $0\leq\varphi\in W_0^{1,\vec\gamma}(\Omega)$,
\begin{equation}\label{singular-one-sided}
\lambda\int_\Omega u^{r-1}\varphi\,dx
\leq \langle\Psi'(u),\varphi\rangle<\infty.
\end{equation}
Choose a smooth $\varphi\in W_0^{1,\vec\gamma}(\Omega)$ that is positive throughout $\Omega$, for instance a positive first Dirichlet eigenfunction of the Laplacian. On $\{u=0\}$ the reaction difference quotient tends to $+\infty$ wherever $\varphi>0$. The finite right-hand side therefore implies $u>0$ a.e. in $\Omega$. Moreover, taking
$\varphi=|\psi|$ in \eqref{singular-one-sided}, we obtain
$$
\lambda\int_\Omega u^{r-1}|\psi|\,dx
\leq C\|\psi\|,
\qquad \forall\psi\in W_0^{1,\vec\gamma}(\Omega).
$$
Since $0<r<1$, the inequality
\(
|(s+t)_+^r-s^r|\leq s^{r-1}|t|,
~s>0,
\)
allows us to apply dominated convergence to the
directional derivatives of $\mathcal I_\lambda$.
Thus, by the minimality of $u$,
$$
\langle\Psi'(u),\psi\rangle
=\lambda\int_\Omega u^{r-1}\psi\,dx,
\qquad \forall\psi\in W_0^{1,\vec\gamma}(\Omega).
$$
Hence $u$ is a weak solution of \eqref{PLR}.

Finally, let $u_1,u_2$ be positive a.e. weak solutions. Their definition
ensures that $u_j^{r-1}|u_1-u_2|$ is integrable, since
$|u_1-u_2|\in W_0^{1,\vec\gamma}(\Omega)$. Subtract the weak equations
and test by $u_1-u_2$. Then
\[
\begin{aligned}
0\leq\langle\mathcal A(u_1)-\mathcal A(u_2),u_1-u_2\rangle=\lambda\int_\Omega(u_1^{r-1}-u_2^{r-1})(u_1-u_2)\,dx\leq0,
\end{aligned}
\]
because $t\mapsto t^{r-1}$ is decreasing. Strict monotonicity in
Proposition \ref{Splus} gives $u_1=u_2$. In particular, the conditional
comparison conclusion also holds, but is weaker than this uniqueness.
\end{proof}

\color{black}
\section{Concave-convex nonlinearities} 
This section extends variational existence results for concave--convex quasilinear equations, such as those in \cite{A2,AlvesHam08}, to the present anisotropic nonhomogeneous setting. We investigate conditions on the functions $f$ and $g$ under which the class of problems
\begin{equation}\label{problem Pfg}\tag{$P_{f,g}$}
\left\{ \begin{array}{ll}
      \mathcal{L}u=\lambda f(u)+g(u)& \mbox{in $\Omega$},\\
        u= 0 & \mbox{on $\partial\Omega$},\end{array} \right.
\end{equation}
has a nontrivial solution. Here the operator is
$$\mathcal{L}u: = -\displaystyle\sum_{i=1}^{N}\displaystyle\partial_i\left(a_i\left(\left|\displaystyle\partial_i u\right|^{p_i}\right)\left|\displaystyle\partial_i u\right|^{p_i- 2}\displaystyle\partial_i u\right), $$
where $\Omega$ is a bounded smooth domain in $\R^N$ and $\lambda$ is a positive parameter.

\subsection{{Small-parameter concave reactions}}
Throughout this subsection assume $\overline\gamma<N$, so $\gamma^*<\infty$. The functions $f$ and $g$ are continuous, and satisfy the following hypotheses:
\begin{itemize}
  \item[($f_1$)] There are positive constants $C_1$ and $C_2$, and $r\in (1,p_1)$ such that
$$C_1t^{r-1}\leq f(t)\leq C_2t^{r-1},\;\forall t>0;$$  

  \item[($g_1$)] There exists a positive constant $C_3$ such that $|g(t)|\leq C_3|t|^{\gamma^*-1}$, $\forall t \in \R$;
  \item[$(g_2)$] There exists $\zeta\in(q_N,\gamma^*]$ such that
\(
0<\zeta G(t)\leq tg(t),~ \forall t>0,
\)
where
\[
G(t)=\int_0^t g(s)\,ds.
\]
\end{itemize}

Since $r>1$ and $\gamma^*>1$, continuity and the growth bounds give $f(0)=g(0)=0$. Thus the positive-part reaction terms are continuously differentiable.
The energy functional $I_\lambda:W_0^{1,\overrightarrow{\gamma}}(\Omega)\to\R$ associated with the problem (\ref{problem Pfg}) is defined by
\begin{equation}\label{functional Alves Hamidi Teo 1}
I_\lambda(u):= \displaystyle\sum_{i=1}^{N}\frac{1}{p_i}\int_{\Omega}A_i\left(\left|\partial_i u\right|^{p_i}\right)dx - \lambda\int_\Omega F(u^+)dx-\int_\Omega G(u^+)dx,
\end{equation}
where
 $$u^+= \max\{u(x), 0\},\qquad A_i(t)=\!\!\displaystyle\int_{0}^{t}\!\!\!a_i(s)ds\qquad\mbox{and}\qquad F(t)=\!\!\!\dis\int_{0}^{t}\!\!\!f(s)ds.$$
To obtain the first solution, we use Ekeland's variational principle and construct a Palais--Smale sequence at a negative level. The following two lemmas establish the required local geometry.

\color{black}
\begin{lemma}\label{alveslema1}
Suppose $(a_1), (a_2), (f_1), (g_1)$ and $(g_2)$. There exist
$\lambda^*>0$ and $b,\rho>0$ such that, for every
$\lambda\in(0,\lambda^*)$, one has $I_\lambda(u)\geq b$ whenever
$\|u\|=\rho$.
\end{lemma}
\begin{proof} Fix $0<\rho<1$ and let $\|u\|=\rho$.
By the inequality (\ref{desig of item a1}), $(f_1)$, $(g_1)$, $(g_2)$ and Sobolev embedding, we obtain
\begin{eqnarray*}
  I_\lambda(u) &\geq&
 \dis\sum_{i=1}^{N}\dis\frac{\alpha_i}{p_i}\left|\dis\partial_i u\right|_{p_i}^{p_i}\!\! +\dis\sum_{i=1}^{N}\left( \dis\frac{H(\zeta_i)\theta_i}{q_i}\left|\dis\partial_i u\right|_{q_i}^{q_i}\right)\!-\lambda\int_\Omega\!F(u^+)dx\!-\!\!\int_\Omega G(u^+)dx. \\
 &\geq& \dis\sum_{i=1}^{N}\dis\frac{\alpha_i}{p_i}\left|\dis\partial_i u\right|_{p_i}^{p_i}+\dis\sum_{i=1}^{N}\dis\frac{H(\zeta_i)\theta_i}{q_i}\left|\dis\partial_i u\right|_{q_i}^{q_i}- \dis\frac{\lambda C_2}{r}\dis\int_\Omega\! |u|^rdx\!-\!\dis\frac{C_3}{\gamma^{*}}\int_\Omega\!|u|^{\gamma^*}dx\\
 &\geq& M_1\left(\dis\sum_{i=1}^{N}\left|\dis\partial_i u\right|_{p_i}^{p_i} +\dis\sum_{i=1}^{N}\dis H(\zeta_i)\left|\dis\partial_i u\right|_{q_i}^{q_i}\right)-\dis\frac{\lambda C_2}{r}\dis\int_\Omega\! |u|^rdx\!-\!\dis\frac{C_3}{\gamma^{*}}\int_\Omega\!|u|^{\gamma^*}dx\\
 &\geq& \widetilde M_1\|u\|^{\gamma_+}
-\lambda M_2\|u\|^{r}-M_3\|u\|^{\gamma^*}.
\end{eqnarray*}
First choose $\rho$ so that
$M_3\rho^{\gamma^*}\leq\widetilde M_1\rho^{\gamma_+}/4$.
Then choose $\lambda^*>0$ so that
$\lambda^*M_2\rho^r\leq\widetilde M_1\rho^{\gamma_+}/4$.
It follows that $I_\lambda(u)\geq
\widetilde M_1\rho^{\gamma_+}/2=:b>0$ for
$0<\lambda<\lambda^*$.
\end{proof}

\color{black}

\color{black}
\begin{lemma}\label{alveslema2} 
The functional $I_\lambda$ is bounded from below on
$\overline{B}_\rho(0)=\{u\in W_0^{1,\overrightarrow{\gamma}}(\Omega): \|u\|\le\rho\}.$
Moreover,
$$K=\displaystyle\inf_{u\in B_\rho(0)}I_\lambda(u)<0, \forall\lambda\in(0,\lambda^*).$$
\end{lemma}
\begin{proof}
The growth assumptions and continuous embeddings give a lower bound for $I_\lambda$ on $\overline{B}_\rho(0)$. Fix a nonzero nonnegative $v\in C_c^\infty(\Omega)$. By \eqref{desig of item a1}, $(f_1)$ and $G(v)\geq0$,
$$I_{\lambda}(tv)\leq \dis\sum_{i=1}^{N}\frac{\beta_it^{p_i}}{p_i}\left|\dis\partial_i v\right|_{p_i}^{p_i}+\dis\sum_{i=1}^{N}\dis H(\zeta_i)\frac{t^{q_i}\theta_i }{q_i}\left|\dis\partial_i v\right|_{q_i}^{q_i}\! -\dis\frac{\lambda C_1 t^r}{r}\dis\int_\Omega\!|v|^rdx.$$
Since $r<p_1<\gamma^*$, it follows from the last inequality that there exists $t_1>0$ small enough such that
$$
I_\lambda(t_1v)<0
\quad\mbox{and}\quad
t_1v\in B_\rho(0).
$$
Therefore, we have reached the conclusion of the lemma.
\end{proof}
Since $K<0$ and $I_\lambda(u)\geq b>0$ on $\partial B_\rho(0)$, we have
\[
K=\inf_{\overline B_\rho(0)}I_\lambda
=\inf_{B_\rho(0)}I_\lambda.
\]
Applying Ekeland's Variational Principle to the functional $I_\lambda$ on the metric space $(\overline{B}_\rho(0), d)$ with the induced metric $d(u,v)=\|u-v\|$, there exists $(u_n)_n\subset\overline{B}_\rho(0)$ such that
$$I_\lambda(u_n)\to K=\dis\inf_{u\in B_\rho(0)}I_\lambda(u)<0$$
and
\begin{equation}\label{Alves Hamidi 2.4}
I_\lambda(v)-I_\lambda(u_n)\geq -\frac{1}{n}\|v-u_n\|,\qquad v\in\overline B_\rho(0).
\end{equation}
Since $I_\lambda(u_n)\to K<0$ whereas $I_\lambda\geq b>0$ on the boundary sphere, $u_n$ lies in $B_\rho(0)$ for all sufficiently large $n$. For each $h$ in the working space, insert $v=u_n\pm th$ in \eqref{Alves Hamidi 2.4}, with $t>0$ small enough. Differentiability and $t\downarrow0$ give $|I_\lambda'(u_n)h|\leq\|h\|/n$. Thus
\begin{equation}\label{Alves Hamidi 2.5}
I'_\lambda(u_n) \to 0.
\end{equation}
From (\ref{Alves Hamidi 2.4}) and (\ref{Alves Hamidi 2.5}), we can conclude
$$I_\lambda(u_n)\to K\quad\mbox{and}\quad I_\lambda'(u_n)\to 0.$$
It follows that $(u_n)_n$ is a bounded $(PS)_{K}$ sequence for the functional $I_\lambda$. Passing to a subsequence, we have $u_n\rightharpoonup u_\lambda$ in $W_0^{1,\overrightarrow{\gamma}}(\Omega)$. 

Moreover, testing $I_\lambda'(u_n)$ with $-u_n^-$ and using
$I_\lambda'(u_n)\to0$, we obtain
\[
\sum_{i=1}^{N}\int_\Omega
a_i\left(|\partial_i u_n^-|^{p_i}\right)
|\partial_i u_n^-|^{p_i}\,dx=o_n(1).
\]
Hence, by $(a_1)$,
\(
u_n^-\to0
~\mbox{in }W_0^{1,\vec{\gamma}}(\Omega).
\)
Consequently, the weak limit satisfies $u_\lambda\geq0$ a.e. in $\Omega$.

\begin{theorem}
Suppose $(a_1), (a_2), (a_3), (f_1), (g_1)$ and $(g_2)$. Then, for every
$\lambda\in(0,\lambda^*)$, problem \eqref{problem Pfg} has a nontrivial
nonnegative weak solution $u_\lambda$.
\end{theorem}
\begin{proof}
Consider the positive constant $\zeta$ according to the property $(g_2)$. It follows
$$I_\lambda(u_n)-\frac{1}{\zeta}I_\lambda'(u_n)u_n= K+o_n(1).$$

Since the functional is defined in terms of $u^+$, we have
\begin{eqnarray}\label{des-1}
\displaystyle\sum_{i=1}^{N}\frac{1}{p_i}
\int_{\Omega}A_i\left(
\left|\dis\partial_i u_n\right|^{p_i}
\right)dx
&-&
\lambda\int_\Omega F(u_n^+)\,dx
-\int_\Omega G(u_n^+)\,dx
\nonumber\\
&-&
\dis\frac{1}{\zeta}\dis\sum_{i=1}^{N}
\int_\Omega
a_i\left(
\left|\dis\partial_i u_n\right|^{p_i}
\right)
\left|\dis\partial_i u_n\right|^{p_i}\,dx
\nonumber\\
&+&
\dis\frac{\lambda}{\zeta}
\int_\Omega f(u_n^+)u_n^+\,dx
+
\dis\frac{1}{\zeta}
\int_\Omega g(u_n^+)u_n^+\,dx
\nonumber\\
&=&K+o_n(1).
\end{eqnarray}

By \eqref{desi-f} we obtain
$$a_i(t^{p_i})t^{p_i}\leq\dfrac{q_i}{p_i}A_i(t^{p_i}),\qquad t\geq0.$$
From this and \eqref{des-1} we deduce that

\[
\begin{aligned}
K+o_n(1)&\geq
\sum_{i=1}^{N}
\left(
\frac{1}{p_i}-\frac{q_i}{\zeta p_i}
\right)
\int_{\Omega}
A_i\left(
\left|
\partial_i u_n
\right|^{p_i}
\right)dx
-\lambda\frac{C_2}{r}
\int_{\Omega}(u_n^+)^{r}dx
\\
&+\frac{1}{\zeta}
\int_{\Omega}
\left(
u_n^+g(u_n^+)-\zeta G(u_n^+)
\right)dx .
\end{aligned}
\]

By the choice of $\zeta$ and assumption $(g_2)$, we obtain

\begin{equation}\label{Alves Hamidi 2.7}
K+o_n(1)\geq
-\frac{C_2\lambda}{r}\|u_n^+\|_r^r,
\qquad n\in\N.
\end{equation}

Since
\(
u_n\to u_\lambda
~\mbox{in }L^r(\Omega)
\)
and $u_\lambda\geq0$, we also have
\(
u_n^+\to u_\lambda
~\mbox{in }L^r(\Omega).
\)
Passing to the limit in \eqref{Alves Hamidi 2.7}, we obtain

$$\dis\frac{-C_2\lambda}{r}\|u_\lambda\|_r^r\leq K<0.$$
Thus, $u_\lambda\neq 0$.

We give the truncation argument explicitly, in the spirit of
\cite{HamRako05}. Set $w_n=u_n-u_\lambda$ and
$b_i(t)=a_i(|t|^{p_i})|t|^{p_i-2}t$. For $\varepsilon>0$, let
$T_\varepsilon(t)=\max\{-\varepsilon,\min\{t,\varepsilon\}\}$.
The sequence $T_\varepsilon(w_n)$ is bounded in the working space and
tends to zero in $L^1(\Omega)$, so it converges weakly to zero there.
Furthermore,
\[
\sup_n\|\lambda f(u_n^+)+g(u_n^+)\|_1<\infty
\]
by $(f_1)$, $(g_1)$ and boundedness in $L^{\gamma^*}(\Omega)$.
Testing the approximate variational equation with $T_\varepsilon(w_n)$
and subtracting $\langle\mathcal A(u_\lambda),T_\varepsilon(w_n)\rangle=o_n(1)$
gives
\[
\int_{\{|w_n|<\varepsilon\}}D_n\,dx\leq C\varepsilon+o_n(1),
\quad
D_n=\sum_i\bigl[b_i(\partial_i u_n)-b_i(\partial_i u_\lambda)\bigr]
(\partial_i u_n-\partial_i u_\lambda)\geq0.
\]
The error $o_n(1)$ is valid because these test functions have uniformly
bounded norms. Fix $\tau>0$ and $R>\tau$. Continuity and strict
monotonicity of the fluxes imply that
\[
\sum_i[b_i(\xi_i)-b_i(\eta_i)](\xi_i-\eta_i)\geq d_{R,\tau}>0
\]
on the compact set $|\xi|,|\eta|\leq R$, $|\xi-\eta|\geq\tau$.
The measure of the sets where $|\nabla u_n|>R$ or
$|\nabla u_\lambda|>R$ is at most $CR^{-\gamma_-}$, by the bounded
directional norms. Consequently,
\[
\begin{aligned}
|\{|\nabla u_n-\nabla u_\lambda|\geq\tau\}|
&\leq |\{|w_n|\geq\varepsilon\}|
+\frac{C\varepsilon+o_n(1)}{d_{R,\tau}}+CR^{-\gamma_-}.
\end{aligned}
\]
First let $n\to\infty$, then $\varepsilon\downarrow0$, and finally
$R\to\infty$. Since $w_n\to0$ in measure, the gradients converge in
measure. Extracting a subsequence, we obtain
\begin{equation}\label{Convergence of gradients}
\nabla u_n(x)\to\nabla u_\lambda(x)\quad\mbox{a.e. in }\Omega.
\end{equation}
No Euler equation for $u_\lambda$ has been assumed in this argument.

For each $i$, by $(a_1)$,
\[
\begin{aligned}
\left|
a_i(|\partial_i u_n|^{p_i})
|\partial_i u_n|^{p_i-2}\partial_i u_n
\right|
\leq
\beta_i|\partial_i u_n|^{p_i-1}
+
H(\zeta_i)\theta_i|\partial_i u_n|^{q_i-1}.
\end{aligned}
\]
Hence, since $(u_n)$ is bounded in
$W_0^{1,\vec{\gamma}}(\Omega)$, the corresponding fluxes are bounded in
$L^{\gamma_i'}(\Omega)$. Moreover, \eqref{Convergence of gradients} and
the continuity of $a_i$ give their pointwise convergence. Therefore,
\[
a_i(|\partial_i u_n|^{p_i})
|\partial_i u_n|^{p_i-2}\partial_i u_n
\rightharpoonup
a_i(|\partial_i u_\lambda|^{p_i})
|\partial_i u_\lambda|^{p_i-2}\partial_i u_\lambda
\]
in $L^{\gamma_i'}(\Omega)$ for every $i$.

Since $r<\gamma^*$, the compact embedding gives
\[
u_n^+\to u_\lambda
\quad\mbox{in }L^r(\Omega),
\]
and consequently
\[
f(u_n^+)\to f(u_\lambda)
\quad\mbox{in }L^{r'}(\Omega).
\]
Furthermore, $(u_n^+)$ is bounded in $L^{\gamma^*}(\Omega)$ and
$u_n^+\to u_\lambda$ a.e. in $\Omega$. Thus, by $(g_1)$,
$(g(u_n^+))$ is bounded in $L^{(\gamma^*)'}(\Omega)$ and
\[
g(u_n^+)\rightharpoonup g(u_\lambda)
\quad\mbox{in }L^{(\gamma^*)'}(\Omega).
\]

Passing to the limit in $I_\lambda'(u_n)\to0$, we obtain
\[
I_\lambda'(u_\lambda)\varphi=0,
\qquad
\forall\varphi\in W_0^{1,\vec{\gamma}}(\Omega).
\]
Consequently, $u_\lambda$ is a nontrivial nonnegative weak solution.

\end{proof}
\subsection{{Large-parameter superlinear reactions with critical growth}}
Here $(f_2)$--$(f_4)$ replace $(f_1)$ and are assumed simultaneously
throughout this subsection. We again assume $\overline\gamma<N$ and
continuity of $f$ on $[0,\infty)$. Define $\kappa=\gamma_+$ and consider
the following assumptions:
\begin{itemize}
 \item[$(f_2)$]  There are constants $\theta\in (q_N,\gamma^*)$ and $t_0>0$ such that $0<\theta F(t)\leq f(t)t,\;\forall t\geq t_0$.
\item[$(f_3)$] There exists $r\in(2, \gamma^*)$ such that
  $$\limsup_{t\to+\infty}\dis\frac{f(t)}{t^{r-1}}<+\infty$$
 \item[$(f_4)$] With the same $\theta$ and $t_0$ as in $(f_2)$,
 \[
 0<\theta F(t)\leq t f(t)\quad\mbox{for }0<t<t_0,
 \qquad
 \lim_{t\to0^+}\frac{f(t)}{t^{\kappa-1}}=0.
 \]
\end{itemize}

In particular, continuity and $(f_4)$ give $f(0)=0$, so the positive-part energy is $C^1$. The second part of $(f_4)$ supplies the required control at the origin, while its first part extends the Ambrosetti--Rabinowitz inequality in $(f_2)$ to the interval $(0,t_0)$.

\color{black}
We now take $g(t)=|t|^{\gamma^*-2}t$ and consider
\[
 I_\lambda(u)=\sum_{i=1}^N\frac1{p_i}\int_\Omega
 A_i(|\partial_i u|^{p_i})\,dx
 -\lambda\int_\Omega F(u^+)\,dx
 -\frac1{\gamma^*}\int_\Omega(u^+)^{\gamma^*}\,dx.
\]

\begin{theorem}\label{critical-concave-convex}
Assume that $\overline{\gamma}<N$, $(a_1)$--$(a_3)$ and
$(f_2)$--$(f_4)$ hold. Then there exists
$\lambda^*>0$ such that problem \eqref{problem Pfg}, with
$g(t)=|t|^{\gamma^*-2}t$, has a nontrivial nonnegative weak solution for
every $\lambda\geq\lambda^*$.
\end{theorem}
\begin{proof}
We divide the proof into three steps.

\emph{Step 1: mountain-pass geometry.}
Conditions $(f_3)$ and $(f_4)$ imply that, for every $\varepsilon>0$, there
is $C_\varepsilon>0$ such that
\[
 |f(t)|\leq\varepsilon |t|^{\kappa-1}
 +C_\varepsilon |t|^{r-1},\qquad t\geq0.
\]

Notice that $(f_2)$ and $(f_4)$ yield
\(
0<\theta F(t)\leq tf(t),~t>0.
\)
Hence,
\[
\frac{d}{dt}\left(\frac{F(t)}{t^\theta}\right)
=
\frac{tf(t)-\theta F(t)}{t^{\theta+1}}
\geq0,
\]
and therefore, for some $C>0$,
\(
F(t)\geq Ct^\theta
\)
for all sufficiently large $t$. On the other hand, $(f_3)$ implies
$f(t)\leq C_1t^{r-1}$ for all sufficiently large $t$, and consequently
$F(t)\leq C_2t^r$. Thus $\theta\leq r$. Moreover, since
$\gamma_i\in\{p_i,q_i\}$ and $p_i<q_i\leq q_N$, we have
$\kappa\leq q_N$. Consequently,
\[
r\geq\theta>q_N\geq\kappa.
\]

By $(a_1)$ and the Sobolev
embeddings, for $\|u\|\leq1$,
\[
 I_\lambda(u)\geq c_0\|u\|^\kappa
 -\lambda\varepsilon C_1\|u\|^\kappa
 -\lambda C_\varepsilon C_2\|u\|^r
 -C_3\|u\|^{\gamma^*}.
\]
For fixed $\lambda>0$, first choose $\varepsilon$ and then $\rho>0$ so
that $I_\lambda(u)\geq b>0$ on $\|u\|=\rho$. On the other hand, if
$0\ne\varphi\in C_c^\infty(\Omega)$ is nonnegative, $(a_1)$ and
$\gamma^*>\kappa$ give $I_\lambda(t\varphi)\to-\infty$ as
$t\to\infty$. Hence $I_\lambda$ has the mountain-pass geometry. Denote its
minimax level by $c_\lambda>0$ and let $(u_n)$ be a corresponding
Palais--Smale sequence.

\emph{Step 2: a compactness level independent of $\lambda$.}
For brevity, put
\[
 \mathcal A_i(u)=\frac1{p_i}\int_\Omega
 A_i(|\partial_i u|^{p_i})\,dx,
 \quad
 \mathcal B_i(u)=\int_\Omega
 a_i(|\partial_i u|^{p_i})|\partial_i u|^{p_i}\,dx.
\]
By \eqref{desi-f},
$\mathcal B_i(u)\leq q_i\mathcal A_i(u)
\leq q_N\mathcal A_i(u)\leq\theta\mathcal A_i(u)$. The outer inequalities are strict when the corresponding energy is positive.
Since $(f_2)$ and $(f_4)$ give
$0<\theta F(t)\leq tf(t)$ for every $t>0$, we obtain
\begin{align*}
 I_\lambda(u_n)-\frac1\theta I_\lambda'(u_n)u_n
 &\geq \sum_{i=1}^N\left(1-\frac{q_i}{\theta}\right)
 \mathcal A_i(u_n)
 +\left(\frac1\theta-\frac1{\gamma^*}\right)
 \|u_n^+\|_{\gamma^*}^{\gamma^*}.
\end{align*}

Since $(u_n)$ is a Palais--Smale sequence at level $c=c_\lambda$, we have
\[
I_\lambda(u_n)=c+o_n(1)
\quad\mbox{and}\quad
I_\lambda'(u_n)u_n=o_n(1)\|u_n\|.
\]
Consequently,
\[
c+o_n(1)+o_n(1)\|u_n\|
\geq
\sum_{i=1}^N\left(1-\frac{q_i}{\theta}\right)
\mathcal A_i(u_n).
\]
Since $\theta>q_N$, the right-hand side is bounded below by $c\|u_n\|^{p_1}-C$ for $\|u_n\|\geq1$, with $c>0$. Since $p_1>1$, it dominates the term $o_n(1)\|u_n\|$ on the left. Therefore $(u_n)$ is bounded in
$W_0^{1,\vec{\gamma}}(\Omega)$.

Moreover, testing $I_\lambda'(u_n)$ with $-u_n^-$, we obtain
\[
\sum_{i=1}^N\int_\Omega
a_i(|\partial_i u_n^-|^{p_i})
|\partial_i u_n^-|^{p_i}\,dx=o_n(1).
\]
Thus, by $(a_1)$,
\(
u_n^-\to0
\quad\mbox{in }W_0^{1,\vec{\gamma}}(\Omega).
\)

We now justify the localized concentration argument. Since
$\gamma_+=\kappa\leq q_N<\theta<\gamma^*$, after extraction
$u_n\rightharpoonup u$ in the working space and $u_n\to u$ a.e. and in
every $L^s(\Omega)$, $1\leq s<\gamma^*$. The negative-part estimate
gives $u\geq0$. Extend $u_n$ and $u$ by zero to $\R^N$ and extract the
finite measures
\[
|\partial_i u_n|^{\gamma_i}\,dx\stackrel{*}{\rightharpoonup}\mu_i,
\qquad (u_n^+)^{\gamma^*}\,dx\stackrel{*}{\rightharpoonup}\nu,
\]
whose supports lie in $\overline\Omega$. There is a finite or countable
set of points $x_j\in\overline\Omega$ such that
\begin{equation}\label{critical-measures}
\nu=u^{\gamma^*}\,dx+\sum_j\nu_j\delta_{x_j},
\qquad
\nu_j^{1/\gamma^*}\leq C_T
\prod_{i=1}^N\mu_{i,j}^{1/(N\gamma_i)},
\quad\mu_{i,j}:=\mu_i(\{x_j\}).
\end{equation}
Here $C_T$ is the Troisi constant for $\vec\gamma$. For completeness,
set $w_n=u_n-u$ and extract the defect measures
$|w_n|^{\gamma^*}\,dx\stackrel{*}{\rightharpoonup}\widehat\nu$ and
$|\partial_i w_n|^{\gamma_i}\,dx\stackrel{*}{\rightharpoonup}\sigma_i$.
The Brezis--Lieb identity and $u_n^-\to0$ give
$\nu=u^{\gamma^*}\,dx+\widehat\nu$. Applying Troisi to $\phi w_n$,
with $\phi\in C_c^\infty(\R^N)$, and using $w_n\to0$ in
$L^{\gamma_i}(\Omega)$, yields
\[
\left(\int|\phi|^{\gamma^*}\,d\widehat\nu\right)^{1/\gamma^*}
\leq C_T\prod_i
\left(\int|\phi|^{\gamma_i}\,d\sigma_i\right)^{1/(N\gamma_i)}.
\]
Approximation of indicators with respect to these finite measures gives,
for every Borel set $E$ and $\sigma=\sum_i\sigma_i$,
\[
\widehat\nu(E)^{1/\gamma^*}\leq C_T\sigma(E)^{1/\overline\gamma}.
\]
On the nonatomic part of $\sigma$, partition into sets of $\sigma$-mass
at most $\varepsilon$. Summing the preceding bound gives a total
$\widehat\nu$-mass at most
$C_T^{\gamma^*}\varepsilon^{\gamma^*/\overline\gamma-1}\sigma(\R^N)$,
which tends to zero. Thus $\widehat\nu$ is supported on the countable
atoms of $\sigma$. Localized Minkowski inequalities, followed by
shrinking neighborhoods and absolute continuity of
$|\partial_i u|^{\gamma_i}\,dx$, show that
$\sigma_i(\{x_j\})=\mu_i(\{x_j\})$. Localizing the product estimate
then proves \eqref{critical-measures}. This is the measure argument
underlying anisotropic concentration--compactness; see also \cite{HamRak07}.

We next relate each atom to the variational equation, including atoms
on $\partial\Omega$. Set
\[
\beta_i^c:=\frac{\gamma^*}{\gamma_i}-1>0,
\qquad \sum_i\beta_i^c=\gamma^*.
\]
Choose $0\leq\eta\leq1$ in $C_c^\infty(\R^N)$, equal to one on
$[-1,1]^N$ and zero outside $[-2,2]^N$, and define
\[
\eta_\varepsilon(x)=\eta\left(
\frac{x_1-x_{j,1}}{\varepsilon^{\beta_1^c}},\ldots,
\frac{x_N-x_{j,N}}{\varepsilon^{\beta_N^c}}\right).
\]
The test $u_n\eta_\varepsilon$ belongs to the working space even for
$x_j\in\partial\Omega$, because $u_n$ has zero trace. For fixed
$\varepsilon$ its norm is bounded independently of $n$, and hence
\[
\begin{aligned}
\sum_i\int\eta_\varepsilon b_i(\partial_i u_n)\partial_i u_n\,dx
&=\int\eta_\varepsilon(u_n^+)^{\gamma^*}\,dx
+\lambda\int\eta_\varepsilon f(u_n^+)u_n^+\,dx\\
&-\sum_i\int b_i(\partial_i u_n)u_n\partial_i\eta_\varepsilon\,dx
+o_n(1).
\end{aligned}
\]
To estimate the mixed terms, write
$Q_\varepsilon=\{|x_i-x_{j,i}|<2\varepsilon^{\beta_i^c}\ \forall i\}$
and let $R_\varepsilon$ be the shell obtained by removing the box with
all factors 2 replaced by 1. Every derivative of $\eta_\varepsilon$
is supported there. By H\"older's inequality,
\[
\begin{aligned}
\int_{R_\varepsilon}|\partial_i u_n|^{m-1}|u_n|
|\partial_i\eta_\varepsilon|\,dx
&\leq C\varepsilon^{-\beta_i^c}
\|\partial_i u_n\|_{L^m(R_\varepsilon)}^{m-1}
\|u_n\|_{L^{\gamma^*}(R_\varepsilon)}
|R_\varepsilon|^{1/m-1/\gamma^*},
\end{aligned}
\]
where $m=p_i$ or, for $i\in\mathscr I_q$, $m=q_i$.
Because $|Q_\varepsilon|=C\varepsilon^{\gamma^*}$, the power of
$\varepsilon$ on the right is
$\gamma^*/m-1-\beta_i^c=\gamma^*(1/m-1/\gamma_i)\geq0$.
It is zero for the effective phase and positive for a lower active
phase. The derivative norms are bounded. Moreover,
\[
\lim_{\varepsilon\downarrow0}\limsup_{n\to\infty}
\int_{R_\varepsilon}|u_n|^{\gamma^*}\,dx=0.
\]
Indeed, choose a sequence of shrinking boxes with boundaries of zero
$\nu$-measure; the shell excludes the central atom and its remaining
$\nu$-mass tends to zero. Since $u_n^-\to0$, the same conclusion holds
for $|u_n|^{\gamma^*}$. The mixed terms therefore vanish in this
iterated limit by $(a_1)$. The subcritical growth of $f$ and strong
$L^r$ convergence imply that the $f$-term also vanishes as the box shrinks.
By $(a_1)$ the energy density is at least
$\sum_i d_i|\partial_i u_n|^{\gamma_i}$, where
\[
d_i=\begin{cases}\alpha_i,&i\notin\mathscr I_q,\\
\theta_i,&i\in\mathscr I_q.
\end{cases}
\]
Passing first $n\to\infty$ and then $\varepsilon\downarrow0$ yields
\begin{equation}\label{critical-atom-balance}
\sum_i d_i\mu_{i,j}\leq\nu_j.
\end{equation}
If $\nu_j>0$, combining \eqref{critical-measures} and
\eqref{critical-atom-balance} gives
\[
\nu_j^{1/\gamma^*}\leq K\nu_j^{1/\overline\gamma},
\qquad K=C_T\prod_i d_i^{-1/(N\gamma_i)}.
\]
Since $1/\overline\gamma-1/\gamma^*=1/N$, every nonzero atom satisfies
$\nu_j\geq\delta:=K^{-N}>0$. Set
\[
c_*:=\left(\frac1\theta-\frac1{\gamma^*}\right)\delta>0.
\]
The earlier energy inequality and boundedness of $(u_n)$ imply
\[
c=\lim_n\left(I_\lambda(u_n)-\frac1\theta I_\lambda'(u_n)u_n\right)
\geq\left(\frac1\theta-\frac1{\gamma^*}\right)\nu_j\geq c_*
\]
whenever an atom is present. Thus no atom occurs if $0<c<c_*$.
By \eqref{critical-measures} and the Brezis--Lieb identity,
$u_n^+\to u$ in $L^{\gamma^*}(\Omega)$. Together with $u_n^-\to0$,
this gives $u_n\to u$ in that space. By $(f_3)$--$(f_4)$ and the
compact subcritical embeddings, the $f$-term converges strongly in the
dual space as well. Testing $I_\lambda'(u_n)\to0$ by $u_n-u$ now gives
\[
\limsup_n\langle\mathcal A(u_n),u_n-u\rangle\leq0.
\]
Proposition \ref{Splus} proves strong convergence in the working space.
Consequently $I_\lambda$ satisfies $(PS)_c$ for every $0<c<c_*$,
without assuming a compact embedding at $\gamma^*$.

\emph{Step 3: the minimax level is below $c_*$.} For the function $\varphi$ fixed above, the map
$t\mapsto I_\lambda(t\varphi)$ is positive for some $t>0$ and tends
to $-\infty$ as $t\to\infty$. Hence it attains its positive global
maximum at some $t_\lambda>0$. The one-dimensional Euler identity is
\[
 \sum_{i=1}^N\mathcal B_i(t_\lambda\varphi)
 =\lambda\int_\Omega f(t_\lambda\varphi)t_\lambda\varphi\,dx
 +t_\lambda^{\gamma^*}\|\varphi\|_{\gamma^*}^{\gamma^*}.
\]

By $(a_1)$, for $t_\lambda\geq1$,
\[
\sum_{i=1}^N\mathcal B_i(t_\lambda\varphi)
\leq C\left(t_\lambda^\kappa+t_\lambda^{p_1}\right)
\leq C_1t_\lambda^\kappa.
\]
Since $\kappa<\gamma^*$, the preceding identity shows that
$(t_\lambda)$ is bounded.

We claim that $t_\lambda\to0$ as $\lambda\to\infty$. Otherwise, there
would exist $\varepsilon>0$ and a sequence $\lambda_n\to\infty$ such that
$t_{\lambda_n}\geq\varepsilon$. Since $(t_{\lambda_n})$ is bounded, up to
a subsequence,
\[
t_{\lambda_n}\to t_0>0.
\]
Since $(t_{\lambda_n})$ is bounded, $\varphi\in C_c^\infty(\Omega)$,
and $f$ is continuous, the dominated convergence theorem yields
\[
\int_\Omega f(t_{\lambda_n}\varphi)t_{\lambda_n}\varphi\,dx
\longrightarrow
\int_\Omega f(t_0\varphi)t_0\varphi\,dx.
\]
Moreover, since $\varphi\geq0$, $\varphi\not\equiv0$, and
$sf(s)>0$ for every $s>0$, we have
\[
\int_\Omega f(t_0\varphi)t_0\varphi\,dx>0.
\]
Therefore the right-hand side of the one-dimensional Euler identity tends
to $+\infty$, whereas its left-hand side remains bounded, a contradiction.
Thus $t_\lambda\to0$.

Choose $T>t_\lambda$ such that $I_\lambda(T\varphi)<0$.
The path $s\mapsto sT\varphi$ is admissible in the mountain-pass
class, and therefore
\[
c_\lambda
\leq \max_{0\leq t\leq T}I_\lambda(t\varphi)
= I_\lambda(t_\lambda\varphi).
\]
Consequently,
\[
0<c_\lambda\leq I_\lambda(t_\lambda\varphi)
\leq\sum_{i=1}^N\mathcal A_i(t_\lambda\varphi)\longrightarrow0.
\]
Choose $\lambda^*>0$ so that $c_\lambda<c_*$ for
$\lambda\geq\lambda^*$. The Palais--Smale sequence then converges strongly
to a critical point $u_\lambda$. Since $c_\lambda>0$, this point is
nontrivial. Testing the Euler equation with $u_\lambda^-$ gives
$u_\lambda\geq0$, and the proof is complete.
\end{proof}

\color{black}

\color{black}

\color{black}
\end{document}